\documentclass[smallextended]{svjour3}       
\smartqed  
\usepackage{graphicx}
\usepackage{fullpage}
\usepackage[numbers]{natbib}
\usepackage{hyperref}       
\usepackage{url}            
\usepackage{booktabs}       
\usepackage{amsfonts}       
\usepackage{nicefrac}       
\usepackage{microtype}      
\usepackage{amsmath,amssymb}
\usepackage{graphicx}
\usepackage{subcaption}
\usepackage{caption}
\usepackage{enumitem}
\usepackage{color}
\usepackage[normalem]{ulem}

\usepackage[ruled,vlined,linesnumbered]{algorithm2e} 
\usepackage{xcolor}
\usepackage{caption} 
\usepackage{subcaption} 

\usepackage{comment}

\usepackage{mathtools}

\usepackage[normalem]{ulem} 

\newcommand{\va}{{\mathbf{a}}}

\newcommand{\vc}{{\mathbf{c}}}
\newcommand{\vd}{{\mathbf{d}}}

\newcommand{\vu}{{\mathbf{u}}}
\newcommand{\vv}{{\mathbf{v}}}

\newcommand{\vx}{{\mathbf{x}}}
\newcommand{\vy}{{\mathbf{y}}}
\newcommand{\vz}{{\mathbf{z}}}

\newcommand{\vC}{{\mathbf{C}}}

\newcommand{\cA}{{\mathcal{A}}}
\newcommand{\cB}{{\mathcal{B}}}

\newcommand{\cL}{{\mathcal{L}}}

\newcommand{\cN}{{\mathcal{N}}}

\newcommand{\cT}{{\mathcal{T}}}

\newcommand{\cX}{{\mathcal{X}}}

\newcommand{\vareps}{\varepsilon}

\newcommand{\EE}{\mathbb{E}} 
\newcommand{\RR}{\mathbb{R}} 
\newcommand{\vzero}{\mathbf{0}} 
\newcommand{\vone}{{\mathbf{1}}} 

\newcommand{\dist}{\mathrm{dist}}    
\newcommand{\prox}{{\mathbf{prox}}} 
\newcommand{\dom}{{\mathrm{dom}}} 

\newcommand{\st}{\mbox{ s.t. }}

\DeclareMathOperator*{\argmin}{arg\,min} 

\newcommand{\bc}{\begin{center}}
\newcommand{\ec}{\end{center}}

\newcommand{\bdm}{\begin{displaymath}}
\newcommand{\edm}{\end{displaymath}}

\newcommand{\beq}{\begin{equation}}
\newcommand{\eeq}{\end{equation}}

\newcommand{\bfl}{\begin{flushleft}}
\newcommand{\efl}{\end{flushleft}}

\newcommand{\bt}{\begin{tabbing}}
\newcommand{\et}{\end{tabbing}}

\newcommand{\beqn}{\begin{eqnarray}}
\newcommand{\eeqn}{\end{eqnarray}}

\newcommand{\beqs}{\begin{align*}} 
\newcommand{\eeqs}{\end{align*}}  

\newtheorem{assumption}{Assumption}

\begin{document}

\title{Stochastic Inexact Augmented Lagrangian Method for Nonconvex Expectation Constrained Optimization}

\author{Zichong Li \and
Pin-Yu Chen \and Sijia Liu \and Songtao Lu \and Yangyang Xu}

\date{\today}

\institute{Z. Li \and Y. Xu \at
		Department of Mathematical Sciences, Rensselaer Polytechnic Institute, Troy, NY, 12180\\
		\email{\{liz19, xuy21\}@rpi.edu}
\and PY. Chen \and S. Lu \at
IBM Research, Thomas J. Watson Research Center, Yorktown Heights, NY 10598\\
\email{\{pin-yu.chen, songtao\}@ibm.com} 
\and S. Liu \at
Department of Computer Science and Engineering, Michigan State University, East Lansing, MI 48824\\
\email{liusiji5@msu.edu}		
}

\maketitle

\begin{abstract}
Many real-world problems not only have complicated nonconvex functional constraints but also use a large number of data points. This motivates the design of efficient stochastic methods on finite-sum or expectation constrained problems. In this paper, we design and analyze stochastic inexact augmented Lagrangian methods (Stoc-iALM) to solve problems involving a nonconvex composite (i.e. smooth+nonsmooth) objective and nonconvex smooth functional constraints.  
We adopt the standard iALM framework and design a subroutine  
by using the momentum-based variance-reduced proximal stochastic gradient method (PStorm) and a postprocessing step. Under certain regularity conditions (assumed also in existing works), to reach an $\varepsilon$-KKT point in expectation, we establish an oracle complexity result of  $O(\varepsilon^{-5})$,  which is better than the best-known $O(\varepsilon^{-6})$ result.  
Numerical experiments on the fairness constrained problem and the Neyman-Pearson classification problem with real data demonstrate that our proposed method outperforms an existing method with the previously best-known complexity result. 
\end{abstract}


\section{Introduction}
	In the big-data era, many real-world applications are dealing with an extremely large amount of data. Many such applications involve nonconvex functional constraints. To compute solutions of these problems, using all data for each update (e.g., in a deterministic method) is prohibitively expensive. This motivates us to design stochastic methods to efficiently compute the solutions.  

	In this paper, we consider the nonconvex expectation constrained problem:
	\begin{equation}\label{eq:ncp-stoc} 
	\begin{aligned}
		&f_0^*:=\min_{\vx\in\RR^d} \big\{f_0(\vx) := g(\vx)+h(\vx), \st \vc(\vx) = \vzero \big\},\\
		&\text{with }~g(\vx) = \EE_{\xi}[G_0(\vx;\xi)],\quad \vc(\vx) = \EE_{\xi}[\vC(\vx;\xi)]\in\RR^m,
	\end{aligned}	
	\end{equation}
	where $h$ is closed convex but possibly nonsmooth, and $\EE_{\xi}$ denotes the expectation taken over the random variable $\xi$.  
Notice that it does not lose generality to use the same random variable $\xi$ in the objective and constraints, because if they depend on two different random variables, we can represent $\xi$ as the stack of the two random variables. We assume that $g(\cdot)$ and $\vc(\cdot)$ are smooth (i.e., the gradient of $g$ and the Jacobian matrix of $\vc$ are Lipschitz continuous) but possibly nonconvex.  
When $\xi$ follows the uniform distribution on $\{1,2,\ldots, N\}$, the problem \eqref{eq:ncp-stoc} reduces to a finite-sum structured problem:
\begin{equation}\label{eq:ncp-finite} 
	\begin{aligned}
		&f_0^*:=\min_{\vx\in\RR^d} \big\{f_0(\vx) := g(\vx)+h(\vx), \st \vc(\vx) = \vzero \big\},\\
		&\text{with }~g(\vx) = \frac{1}{N}\sum_{\xi=1}^N G_0(\vx;\xi),\quad \vc(\vx) = \frac{1}{N}\sum_{\xi=1}^N \vC(\vx;\xi)\in\RR^m,
	\end{aligned}	
	\end{equation}	
which arises from applications involving a large amount of pre-collected data.
	
	Though only equality constraints are included, the formulation \eqref{eq:ncp-stoc} is general enough. As shown in \cite{li2021rate}, 
	an inequality constraint $t(\vx)\le 0$ can be equivalently formulated as an equality constraint $t(\vx)+s=0$ by enforcing the nonnegativity of $s$, and the Karush-Kuhn-Tucker (KKT) conditions of the reformulation are equivalent to those of the original one. Also, a simple convex constraint set $\cX$ can be included in \eqref{eq:ncp-stoc} by setting (part of) $h$ to 
	the indicator function $\vone_{\cX}(\vx)=0$ if $\vx \in \cX$ and $+\infty$ otherwise. 
	Many applications can be formulated to \eqref{eq:ncp-stoc},  such as Neyman-Pearson classification \cite{neyman1933ix, rigollet2011neyman} and the fairness constrained problem \cite{ma2019proximally}.

	Due to the presence of nonconvexity and stochasticity in both objective and constraints, solving \eqref{eq:ncp-stoc} is very challenging. Only a few works (e.g., \cite{ma2019proximally, boob2022stochastic}) have proposed and analyzed methods to solve such a problem. However, no existing methods have fully exploited the structure of \eqref{eq:ncp-stoc}. 
 We will present a stochastic method for \eqref{eq:ncp-stoc} under the general expectation setting, and establish its oracle complexity result, where the oracle can return the function value and gradient of $G_0$ and $\vC$ at any point $\vx$ and a sample of $\xi$. We follow the iALM framework and adopt the momentum-based variance-reduced proximal stochastic gradient method (PStorm) \cite{xu2022momentum} to design a subroutine.

\subsection{Contributions}
Our contributions are two-fold. First, we propose novel stochastic gradient-type methods, based on the framework of the inexact augmented Lagrangian method (Stoc-iALM), for solving nonconvex composite optimization problems with nonlinear nonconvex (but smooth) expectation constraints, in the form of \eqref{eq:ncp-stoc}. By exploiting the so-called mean-squared smoothness structure, we apply PStorm \cite{xu2022momentum} together with a proposed postprocessing step to design a subroutine within the framework of Stoc-iALM.  
The subroutine design is crucial to yield our complexity result that is better than existing best-known results and for good numerical performance, as its complexity has low-order dependence not only on a target error tolerance but also on other quantities such as the smoothness constant, variance bound, and initial objective gap.

Second, we conduct complexity analysis on the proposed Stoc-iALM with the designed subroutine.  
Under a regularity condition (that was also assumed in many existing works \cite{li2021rate, li2022zeroth, lin2022complexity, sahin2019inexact}), we obtain  
an $O(\vareps^{-5})$ oracle complexity result for the expectation-constrained problem \eqref{eq:ncp-stoc}. 
Our $O(\vareps^{-5})$ result yields a substantial improvement over the best-known $\tilde{O}(\vareps^{-6})$ and $O(\vareps^{-6})$ complexity results\footnote{In this paper, we use $\tilde{O}$ to suppress all logarithmic terms of $\vareps$ from the big-$O$ notation.} of the proximal-point methods in \cite{ma2019proximally} and \cite{boob2022stochastic}, which iteratively perturb both the objective and constraints and solve a perturbed convex constrained subproblem.

\subsection{Related Works}
In this subsection, we discuss related works on the inexact augmented Lagrangian method (iALM) and other first-order methods (FOMs) on functional constrained optimization.

The iALM has been popularly used for solving constrained problems. It alternatingly updates the primal variable by approximately minimizing the augmented Lagrangian function and the Lagrangian multiplier (also called dual variable) by dual gradient ascent \cite{hestenes1969multiplier, rockafellar1973dual}. For deterministic convex linear and/or nonlinear constrained problems, the iALM-based FOM in \cite{doi:10.1287/ijoo.2021.0052, lan2016iteration-alm} and the proximal-iALM-based one in \cite{li2021inexact} obtain an $\vareps$-KKT point with $O\left( \vareps^{-1} \log\frac{1}{\vareps} \right)$ gradient evaluations, and the AL-based FOMs in \cite{xu2021iteration, xu2021first, ouyang2015accelerated, li2021inexact, nedelcu2014computational} obtain an $\vareps$-optimal solution with $O(\vareps^{-1})$ gradient evaluations. For strongly-convex problems, the results are reduced to $O\left( \vareps^{-0.5} \log\frac{1}{\vareps} \right)$ and $O(\vareps^{-0.5})$ respectively, e.g., in \cite{doi:10.1287/ijoo.2021.0052, li2021inexact, xu2021iteration, nedelcu2014computational, necoara2014rate}. For deterministic nonconvex problems with nonlinear convex constraints, when Slater's condition holds, $\tilde O(\vareps^{-\frac{5}{2}})$ complexity results are obtained by the AL or penalty based FOMs in \cite{doi:10.1287/ijoo.2021.0052, lin2022complexity} and the proximal ALM-based FOM in \cite{meloiteration2020}. If the constraints are polyhedral and the objective is smooth, the complexity can be reduced to $O(\vareps^{-2})$ with a hidden constant dependent on the so-called Hoffman's bound of the polyhedral set \cite{zhang2022global}. Different from Slater's condition, a regularity condition is assumed in \cite{li2021inexact}, which obtains an $\tilde O(\vareps^{-\frac{5}{2}})$ result by an iALM-based FOM. The regularity condition is used to guarantee near feasibility from near stationarity of the AL function. Assuming a similar regularity condition, \cite{li2021inexact} and \cite{lin2022complexity} both achieve $\tilde O(\vareps^{-3})$ results for deterministic problems with nonconvex constraints, by an iALM based FOM and a proximal-point penalty based FOM respectively.  

There are many papers studying FOMs on convex stochastic constrained problems (e.g., \cite{lan2020algorithms, yan2022adaptive, xu2020primal}). Also, a few papers (e.g., \cite{wang2017penalty, shi2022momentum, jin2022stochastic}) have studied FOMs for nonconvex optimization with stochastic objective but deterministic constraints, either based on an exact-penalty framework or ALM. However, few papers have studied FOMs for the nonconvex expectation constrained problems. On solving inequality expectation constrained nonconvex optimization, both \cite{ma2019proximally} and \cite{boob2022stochastic} design stochastic first-order methods in the framework of the proximal-point (PP) method. They achieve $\tilde{O}(\vareps^{-6})$ and $O(\vareps^{-6})$ complexity results respectively, which are higher than our $O(\vareps^{-5})$ result. Both PP-based methods in \cite{ma2019proximally, boob2022stochastic} iteratively perturb the nonconvex objective and constraint functions to be strongly convex and inexactly solve the constrained convex subproblems. To achieve their results, the PP-based method in \cite{ma2019proximally} uses the online stochastic subgradient subroutine in \cite{yu2017online}, while the one in \cite{boob2022stochastic} designs a constraint extrapolation (ConEx) subroutine.  
Note that nonconvex structures that we assume are different from those in \cite{boob2022stochastic} and \cite{ma2019proximally}. While we assume a nonconvex composite objective and smooth constraints, the method in \cite{boob2022stochastic} applies to nonconvex problems where both the objective and constraint functions can be nonconvex composite, and \cite{ma2019proximally} only assumes weak convexity\footnote{A function $f$ is $\rho$-weakly convex for some $\rho>0$, if $f(\cdot) + \frac{\rho}{2}\|\cdot\|^2$ is convex} on the objective and constraint functions. However, even with the nonconvex structures that we assume, the methods in \cite{ma2019proximally} and \cite{boob2022stochastic} can still only achieve the $\tilde{O}(\vareps^{-6})$ and $O(\vareps^{-6})$ complexity results, as they do not exploit the smoothness structure in their subroutines.

Stochastic FOMs have also been proposed for minimax problems (e.g., \cite{tran2020hybrid, luo2020stochastic, huang2022accelerated}). The work \cite{tran2020hybrid} gives a hybrid variance-reduced stochastic gradient method for nonconvex-linear minimax problems with a compact domain of dual variables and establishes an $O(\vareps^{-5})$ complexity result to find an $\vareps$-stationary point. Although a nonlinear-constrained problem can be formulated as a nonconvex-linear minimax problem by the ordinary Lagrangian function, KKT conditions of the former are stronger than stationarity conditions of the latter that assumes a compact dual domain. This is due to the fact that the stationary point of a nonconvex-concave minimax problem with a compact dual domain may not be primal feasible. Both of \cite{luo2020stochastic, huang2022accelerated} assume strong concavity on the dual side. Let $\kappa$ be the condition number of the dual part. The method in \cite{luo2020stochastic} needs $O(\kappa^3\vareps^{-3})$ sample complexity to produce an $\vareps$-stationary solution, while the complexity result in  \cite{huang2022accelerated} is $\tilde O(\kappa^{\frac{9}{2}}\vareps^{-3})$. In order to obtain an $\vareps$-KKT point of the problem \eqref{eq:ncp-stoc} that we consider, under the regularity condition in Assumption~\ref{assump:reg-stoc} below, we can apply the methods in \cite{luo2020stochastic, huang2022accelerated} to a penalized problem $\min_\vx \big\{f_0(\vx) + \frac{\rho}{2}\|\vc(\vx)\|^2\big\}$ with $\rho = \Theta(\vareps^{-1})$, which is equivalent to the nonconvex strongly-concave minimax problem $\min_\vx\max_\vy \big\{f_0(\vx) + \vy^\top \vc(\vx) - \frac{1}{2\rho}\|\vy\|^2\big\}$. The resulting complexity results will be $O(\vareps^{-6})$ by the method in \cite{luo2020stochastic} and $O(\vareps^{-\frac{15}{2}})$ by the method in \cite{huang2022accelerated}, as the condition number of the equivalent minimax problem is $\Theta(\vareps^{-1})$. 

\subsection{Notations}\label{sec:notation-stoc}
We use $\Vert \cdot \Vert$ for the Euclidean norm of a vector and the spectral norm of a matrix. The notation $[n]$ denotes the set $\{1,\ldots,n\}$. For any $a\in\RR$, $[a]_{1+}:=\max\{a, 1\}$. The natural logarithmic function is $\ln(\cdot)$, and $e=2.71828...$ represents its base. We denote $J_\vc(\vx)$ as the Jacobian matrix of $\vc$ at $\vx$ and $J_\vC(\vx;\xi)$ the Jacobian matrix of $\vC(\,\cdot\,; \xi)$ at $\vx$. The distance between a vector $\vx$ and a set $\cX$ is denoted as $\dist(\vx, \cX) = \min_{\vy \in \cX} \Vert \vx-\vy \Vert$. The proximal operator of a convex function $r$ is defined as $\prox_r(\vx) := \argmin_{\vu}\{r(\vu) + \frac{1}{2}\Vert \vu-\vx \Vert^2 \}$. $\EE_{\xi_1,\xi_2}$ takes expectation about $\xi_1$ and $\xi_2$, and we always assume that $\xi_1$ and $\xi_2$ are independent and follow the same distribution as $\xi$ in \eqref{eq:ncp-stoc}. We use $\partial f$ to denote the subdifferetial of a function $f$.
	The augmented Lagrangian (AL) function of \eqref{eq:ncp-stoc} is
	\begin{equation}\label{eq:AL-stoc}
		\textstyle \cL_{\beta}(\vx,\vy) = f_0(\vx) + \vy^\top \vc(\vx) + \frac{\beta}{2} \Vert \vc(\vx) \Vert^2,
	\end{equation} 
	where $\beta > 0$ is the penalty parameter, and $\vy \in \RR^m$ is the multiplier or the dual variable. 
	
	\begin{definition}[$\vareps$-KKT point in expectation]\label{def:eps-kkt-stoc}
		Given $\vareps \geq 0$, a point $\vx \in \RR^d$ is called an $\vareps$-KKT point in expectation to \eqref{eq:ncp-stoc} if there is a vector $\vy\in\RR^m$ such that 
		\begin{equation*}
			\EE\big[\Vert \vc(\vx) \Vert^2\big] \leq \vareps^2,\quad
			\EE\left[\dist\left(\vzero, \partial f_0(\vx)+J_\vc^\top(\vx) \  \vy \right)^2\right] \leq \vareps^2.
		\end{equation*}
	\end{definition}
	
\section{Stochastic iALM and its outer iteration complexity}\label{sec:stoc-outer-conv}
To efficiently find a near KKT-point of \eqref{eq:ncp-stoc}, 
we design a stochastic gradient-type method based on the framework of the stochastic inexact augmented Lagrangian method (Stoc-iALM), which is given in Algorithm \ref{alg:ialm-stoc}. Because of nonconvexity, we can only produce a near-stationary point of each subproblem, as required in \eqref{eq:ialm-x-stoc}. Though the condition in \eqref{eq:ialm-x-stoc} is not checkable (due to taking expectation), it can be guaranteed from the convergence rate result of the subroutine that we will give in Section \ref{sec:ialm+pstorm-conv}. The update to the multiplier is inspired by \cite{sahin2019inexact, li2021rate} and adapts to the estimated primal infeasibility. With an appropriate choice of $\gamma_k$, we can ensure $\frac{\|\vy^k\|}{\beta_k} \to 0$, which is crucial in our analysis.

	\begin{algorithm}[h] 
		\caption{Stochastic inexact augmented Lagrangian method (Stoc-iALM) for solving \eqref{eq:ncp-stoc}}\label{alg:ialm-stoc}
		\DontPrintSemicolon
		\textbf{Initialization:} given $\vareps>0$, set $\vy^0 = \vzero$ and choose $\vx^0\in\dom(f_0)$, $\beta_0>0$, $\sigma>1$, and an integer sequence $\{M_k\}$\; 
		\For{$k=0,1,\ldots,$}{
			Let $\beta_k=\beta_0\sigma^k$. \;
			Obtain $\vx^{k+1}$ (by a subroutine) satisfying 
			\begin{equation}\label{eq:ialm-x-stoc}
				\EE \left[\dist(\vzero, \partial_x \cL_{\beta_k}(\vx^{k+1},\vy^k))^2\,\big|\, \vy^k\right] \le \vareps^2.
			\end{equation}
			
			Obtain i.i.d samples $\{\xi^k_i\}_{i=1}^{M_k}$ and set $\tilde{\vc}(\vx^{k+1}) = \frac{1}{M_k} \sum_{i=1}^{M_k} \vC(\vx^{k+1},\xi^k_i)$.\;
			Update $\vy$ by
			\begin{align}
				\vy^{k+1} = ~\vy^k  + \min\left\{\beta_k,\ \frac{\gamma_k }{\Vert \tilde{\vc}(\vx^{k+1}) \Vert}\right\} \tilde{\vc}(\vx^{k+1}).\label{eq:alm-y-stoc}
			\end{align}
		}
	\end{algorithm}	
	
	Without specifying a subroutine to obtain $\vx^{k+1}$, we first establish the outer iteration complexity result of Algorithm~\ref{alg:ialm-stoc}, by following the analysis in \cite{li2021rate, lin2022complexity}. 
	Throughout this paper, we make the following assumptions about \eqref{eq:ncp-stoc}.

		\begin{assumption}[stochastic first-order oracle]\label{assump:oracle}
	For the problem \eqref{eq:ncp-stoc}, a stochastic first-order oracle can be accessed. At any $\vx\in\dom(h)$, the oracle can obtain a sample $\xi$ and return $(\nabla G_0(\vx, \xi), \vC(\vx, \xi), J_\vC(\vx, \xi))$. 
	\end{assumption}
	
\begin{remark}
The overall complexity result of our algorithm will be measured by the total number of stochastic first-order oracles that are called. Though the oracle can return a tuple $(\nabla G_0(\vx, \xi), \vC(\vx, \xi), J_\vC(\vx, \xi))$, our algorithm may only use part of it during one update. However, even if part of an oracle is used, one oracle will be counted in measuring the complexity result.
\end{remark}	
	
	\begin{assumption}[structured bounded domain]\label{assump:comp-stoc}
		The domain of $h$, denoted as $\cX:=\dom(h)$, is compact.
		Moreover, for some $M > 0$, it holds that $\partial h(\vx) \subseteq \cN_{\cX}(\vx) + \cB_M, \forall \vx \in \cX$, where $\cN_{\cX}(\vx)$ denotes the normal cone of $\cX$ at $\vx$, and $\cB_M$ denotes a closed ball of radius $M$ centering at the origin.
	\end{assumption}
\begin{remark}
Assumption \ref{assump:comp-stoc} holds for rather general choices of $h(\cdot)$. For example, it holds for any $h(\cdot) := r(\cdot) + \vone_{\cX}(\cdot)$ as long as $\partial r$ is bounded everywhere (e.g., the $\ell_p$-norm for $p\ge1$), where $\vone_{\cX}$ denotes the indicator function on $\cX$. Under Assumption~\ref{assump:comp-stoc}, there must exist finite constants $B_0$ and $B_c$ such that	
\begin{subequations}\label{P3-eqn-33-stoc}
		\begin{align}
			&B_0 \ge \max_{\vx\in\dom(h)}\max \big\{|f_0(\vx)|,\left\Vert \nabla g(\vx) \right\Vert\big\}, \label{P3-eqn-33-stoc-a}\\
			&B_c \ge \max_{\vx\in\dom(h)}\Vert J_\vc(\vx) \Vert. \label{P3-eqn-33-stoc-c}
		\end{align}
	\end{subequations}
\end{remark}

	Due to nonconvexity of the constraints 
	in \eqref{eq:ncp-stoc}, one may not even find a near-feasible point in polynomial time. Therefore, following \cite{li2021rate, li2022zeroth, lin2022complexity, sahin2019inexact}, we assume a regularity condition on the constraints in \eqref{eq:ncp-stoc}, which ensures that a near-stationary point of the AL function is near feasible to \eqref{eq:ncp-stoc}, if the penalty parameter is big. Note that knowledge of $v$ below is not required in Algorithm~\ref{alg:ialm-stoc}.
	\begin{assumption}[regularity condition] \label{assump:reg-stoc} 
		There is a constant $v > 0$ such that for any $\vx \in \cX=\dom(h)$,
		\begin{equation}\label{eq:regularity-stoc}
			\textstyle v\Vert \vc(\vx) \Vert \le \dist \left(-J_\vc(\vx)^{\top} \vc(\vx), \cN_{\cX}(\vx) \right).
		\end{equation}
	\end{assumption} 
	
\begin{remark}
Here, we give a couple of remarks about the regularity condition. First, this regularity condition has been proven for many applications. For example, \cite{li2021rate} shows that it holds for all affine-equality constrained problems with possibly additional polyhedral or ball constraint sets. Other examples are given in \cite{lin2022complexity, sahin2019inexact}. 
Second, to find a near KKT point of a nonconvex expectation constrained problem, the two existing works \cite{ma2020quadratically} and \cite{boob2022stochastic} 
also need a certain regularity condition. 
Different from Assumption~\ref{assump:reg-stoc}, a uniform Slater's condition is assumed in \cite{ma2020quadratically}, and a strong MFCQ condition is assumed in \cite{boob2022stochastic}. Those conditions are neither strictly stronger nor strictly weaker than Assumption \ref{assump:reg-stoc}, as shown in \cite{lin2022complexity}.
\end{remark}

The next lemma will be used to upper bound $\frac{\|\vy^k\|}{\beta_k}$.

\begin{lemma}\label{lem:bd-vy-beta}
For any constants $\alpha >1$ and $\sigma >1$, if $\alpha \ge \frac{8}{\ln \sigma}$ and $\ln\sigma \ge \frac{8}{e^4}$, then it holds $\alpha\ge \log_\sigma \alpha^2$. In addition, for any $x \ge \log_\sigma\alpha^2$, it holds $\frac{\sigma^x}{x}\ge \alpha$.
\end{lemma}

\begin{proof}
Define $\phi(\alpha) = \alpha - \frac{2\ln \alpha}{\ln\sigma}$. Then $\phi'(\alpha) = 1 - \frac{2}{\alpha\ln\sigma} > 0, \forall\, \alpha > \frac{2}{\ln \sigma}$. Hence, $\phi(\cdot)$ is increasing on $(\frac{2}{\ln \sigma}, \infty)$. In addition, the condition $\ln\sigma \ge \frac{8}{e^4}$ implies $\phi(\frac{8}{\ln \sigma}) \ge 0$. Thus, for $\alpha \ge \frac{8}{\ln \sigma}$, it holds $\phi(\alpha)\ge 0$ that is equivalent to $\alpha\ge \log_\sigma \alpha^2$. 

Now define $\psi(x) = \sigma^x - \alpha x$. Then $\psi'(x) = \sigma^x\cdot\ln \sigma - \alpha$, and thus for any $x \ge \log_\sigma\alpha^2$, we have $\psi'(x)\ge \psi'(\log_\sigma\alpha^2)=\alpha^2 \ln\sigma - \alpha \ge 7\alpha >0$, where we have used $\alpha \ge \frac{8}{\ln \sigma}$. Hence, $\psi(\cdot)$ is increasing on $[\log_\sigma\alpha^2, \infty)$, and for any $x \ge \log_\sigma\alpha^2$, it holds $\psi(x) \ge \psi(\log_\sigma\alpha^2) = \alpha^2 - \alpha \log_\sigma \alpha^2\ge0$. This completes the proof.
\qed\end{proof}

The theorem below gives the outer iteration number of Algorithm~\ref{alg:ialm-stoc} to produce an $\vareps$-KKT point in expectation of \eqref{eq:ncp-stoc}.

\begin{theorem}[Outer iteration complexity of Stoc-iALM]\label{thm:stoc-ialm-conv}
In \eqref{eq:alm-y-stoc}, set  
$\gamma_k = \gamma_0,\forall\, k\ge0$ for some $\gamma_0 > 0$ such that  
$\frac{\sqrt{8} B_c \gamma_0}{\beta_0 v \vareps} \ge \frac{8}{\ln\sigma}$.  
Then under Assumptions~\ref{assump:comp-stoc} and \ref{assump:reg-stoc}, Algorithm~\ref{alg:ialm-stoc} needs at most $K$ outer iterations to find an $\vareps$-KKT point in expectation of \eqref{eq:ncp-stoc},  
where  
		\begin{equation} \label{eq:K-stoc}
		K = \max\left\{\left\lceil \log_\sigma \frac{\sqrt{8}\sqrt{\vareps^2 + B_0^2 + M^2}}{\beta_0 v \vareps}\right\rceil, \, \left\lceil 2\log_\sigma \frac{\sqrt{8} B_c \gamma_0}{\beta_0 v \vareps} \right\rceil\right\} + 1.
		\end{equation}
\end{theorem}

\begin{proof}
First, by $\vy^0 = \vzero$, the $\vy$-update in \eqref{eq:alm-y-stoc}, and the choice of $\gamma_k$, we have from the triangle inequality that for any $k \ge 0$,
		\begin{align}
			\Vert \vy^k \Vert &\le \sum_{t=0}^{k-1} \gamma_t  = k \gamma_0, \label{eq:y-bound-stoc}
		\end{align}
where by the convention we define $\sum_{t=0}^{k-1} \gamma_t = 0$ if $k=0$.		
		
Second, from \eqref{eq:regularity-stoc}, we have 
\begin{align}\label{eq:reg-iter}
\EE \left[\Vert \vc(\vx^k) \Vert^2\right] \le & ~\frac{1}{v^2} \EE \left[ \dist \left(-J_\vc(\vx^k)^{\top} \vc(\vx^k), \cN_{\cX}(\vx^k)\right)^2\right]\cr
 = & ~ \frac{1}{v^2 \beta_{k-1}^2} \EE \left[ \dist \left(-\beta_{k-1} J_\vc(\vx^k)^{\top} \vc(\vx^k), \beta_{k-1}\cN_{\cX}(\vx^k)\right)^2\right]\cr
 = & ~ \frac{1}{v^2 \beta_{k-1}^2} \EE \left[ \dist \left(-\beta_{k-1} J_\vc(\vx^k)^{\top} \vc(\vx^k), \cN_{\cX}(\vx^k)\right)^2\right],
\end{align}
where the last equation follows from $\cN_\cX(\vx) = \beta\cdot \cN_\cX(\vx), \forall\, \beta>0, \forall\, \vx\in \cX$. 
In addition, by $\partial h(\vx) \subseteq \cN_{\cX}(\vx) + \cB_M$ from Assumption~\ref{assump:comp-stoc}, it holds $\dist\big(\vz, \cN_{\cX}(\vx) + \cB_M\big) \le \dist\big(\vz, \partial h(\vx)\big)$ for any $\vz\in\RR^d$. Also, it holds from the triangle inequality that $\dist\big(\vz, \cN_{\cX}(\vx) \big) \le \dist\big(\vz, \cN_{\cX}(\vx) + \cB_M\big) + M$. Hence, 
\begin{equation}\label{eq:dist-ineq}
\dist\big(\vz, \cN_{\cX}(\vx) \big) \le \dist\big(\vz, \partial h(\vx)\big) +M, \forall\, \vx\in\cX, \forall\, \vz\in\RR^d.
\end{equation} 
Using \eqref{eq:dist-ineq} with $\vz = - \beta_{k-1} J_\vc(\vx^k)^{\top} \vc(\vx^k)$ and noticing 
$$\partial h(\vx^k)= \partial_x \cL_{\beta_{k-1}}(\vx^{k},\vy^{k-1}) - \nabla g(\vx^{k}) - J_\vc(\vx^k)^{\top} \vy^{k-1} - \beta_{k-1} J_\vc(\vx^k)^{\top} \vc(\vx^k),$$
it holds
\begin{align*}
\dist \left(-\beta_{k-1} J_\vc(\vx^k)^{\top} \vc(\vx^k), \cN_{\cX}(\vx^k) \right) \le \dist \big(\vzero, \partial_x \cL_{\beta_{k-1}}(\vx^{k},\vy^{k-1}) - \nabla g(\vx^{k}) - J_\vc(\vx^k)^{\top} \vy^{k-1} \big) + M,
\end{align*}
which together with \eqref{eq:reg-iter} gives
		\begin{align*}
			&\EE \left[\Vert \vc(\vx^k) \Vert^2\right] \le \frac{1}{v^2 \beta_{k-1}^2} \EE \left[ \dist \big(0, \partial_x \cL_{\beta_{k-1}}(\vx^{k},\vy^{k-1}) - \nabla g(\vx^{k}) - J_\vc(\vx^k)^{\top} \vy^{k-1} \big) + M  \right] ^2.
		\end{align*}
Moreover, applying the triangle inequality to the right hand side of the inequality above, we have
\begin{equation}\label{eq:c-bound-2}
\begin{aligned}
& ~\EE \left[\Vert \vc(\vx^k) \Vert^2\right] \\
\le & ~ \frac{1}{v^2 \beta_{k-1}^2}\EE \left[  \Big(\dist \big(0, \partial_x \cL_{\beta_{k-1}}(\vx^{k},\vy^{k-1}) \big) + \Vert \nabla g(\vx^k) \Vert + \Vert J_\vc(\vx^k) \Vert \Vert \vy^{k-1} \Vert + M  \Big) ^2 \right], \forall\, k\ge1.
\end{aligned}
\end{equation}	
Now, using the Young's inequality and by \eqref{P3-eqn-33-stoc-a}, \eqref{P3-eqn-33-stoc-c}, \eqref{eq:ialm-x-stoc} and \eqref{eq:y-bound-stoc}, we obtain
\begin{equation}\label{eq:c-bound}
\EE [\Vert \vc(\vx^k) \Vert^2] \le \frac{4}{v^2 \beta_{k-1}^2} (\vareps^2 + B_0^2 + B_c^2 (k-1)^2\gamma_0^2 + M^2),\forall\, k\ge1.
\end{equation}
Since $\beta_k = \beta_0\sigma^k$, we have from the choice of $K$ in \eqref{eq:K-stoc} that $\frac{4}{v^2 \beta_{K-1}^2} (\vareps^2 + B_0^2  + M^2) \le \frac{\vareps^2}{2}$. In addition, let $\alpha=\frac{\sqrt{8} B_c \gamma_0}{\beta_0 v \vareps}$. Then from the choice of $K$, it holds $K-1 \ge \log_\sigma \alpha^2$, and thus from Lemma~\ref{lem:bd-vy-beta}, we have $\frac{\sigma^{K-1}}{K-1} \ge \alpha$. Hence,
$$\frac{4B_c^2 (K-1)^2\gamma_0^2}{v^2 \beta_{K-1}^2} = \frac{4B_c^2 (K-1)^2\gamma_0^2}{v^2 \beta_0^2 \sigma^{2(K-1)}} \le \frac{4B_c^2 \gamma_0^2}{v^2 \beta_0^2 \alpha^2}=\frac{\vareps^2}{2}.$$
Thus it follows from \eqref{eq:c-bound} that $\EE [\Vert \vc(\vx^K) \Vert^2] \le \vareps^2$.

		Finally, it holds from \eqref{eq:ialm-x-stoc} that
		\begin{equation*} 
			\EE \left[ \dist \left( \vzero, \partial f_0(\vx^{K}) + J_\vc(\vx^{K})^{\top} \ (\vy^{K-1} + \beta_{K-1}\vc(\vx^{K})) \right)^2 \right] \le \vareps^2.
		\end{equation*}
		Therefore, $\vx^K$ is an $\vareps$-KKT point in expectation of \eqref{eq:ncp-stoc} with the corresponding multiplier $\vy^{K-1} + \beta_{K-1} \vc(\vx^K)$, according to Definition \ref{def:eps-kkt-stoc}.
\qed
\end{proof}

\begin{remark}
A few remarks about Theorem~\ref{thm:stoc-ialm-conv} are as follows. First, the condition $\frac{\sqrt{8} B_c \gamma_0}{\beta_0 v \vareps} \ge \frac{8}{\ln\sigma}$ requires to know $v$. However, this is only for the ease of analysis. We do not actually need the exact value of $v$. Notice that we can assume $v\le 1$, because if \eqref{eq:regularity-stoc} holds for some $v>1$, it also holds with $v=1$. In this case, it suffices to pick $\gamma_0$ such that $\frac{\sqrt{8} B_c \gamma_0}{\beta_0 \vareps} \ge \frac{8}{\ln\sigma}$, which does not involve $v$. Second, for convex cases where $\vc(\cdot)$ consists of affine constraints, it is guaranteed that $\vc(\vx^{k+1}) = O(\frac{1}{\beta_k})$ if \eqref{eq:ialm-x-stoc} holds in a deterministic way (i.e., without the expectation) and the strong duality holds for \eqref{eq:ncp-stoc}; see \cite{xu2021iteration} for example. In this case, with $\gamma_k=\gamma_0,\forall\, k$ for an appropriate $\gamma_0$, the dual update will accept $\beta_k$ as the stepsize for all $k$. Third, the result in Theorem~\ref{thm:stoc-ialm-conv} does not depend on the setting of $\tilde\vc(\vx^{k+1})$. However, the multiplier update by the \emph{classic} ALM is a key to have good practical performance. Hence for the special case in \eqref{eq:ncp-finite}, we set  $\tilde{\vc}(\vx^{k+1})=\vc(\vx^{k+1})$, and for the general problem \eqref{eq:ncp-stoc}, we will choose $M_k=\Theta(\frac{1}{\vareps^2})$ so that $\EE[\|\tilde{\vc}(\vx^{k+1})-\vc(\vx^{k+1})\|^2] = O(\frac{1}{M_k})= O(\vareps^2)$ and thus $\tilde{\vc}(\vx^{k+1})$ will be close to $\vc(\vx^{k+1})$ in expectation.
 Finally, we have not specified the subroutine. To have a low overall complexity in terms of the number of sample/component gradients, it is important to obtain each $\vx^{k+1}$ efficiently. In the next section, we will exploit the problem structure and design an efficient subroutine to make \eqref{eq:ialm-x-stoc} hold for each $k$. 
\end{remark}


\section{Momentum-accelerated subroutine and overall oracle complexity} \label{sec:ialm+pstorm-conv}
In this section, we give a subroutine to find each $\vx^{k+1}$ in Algorithm~\ref{alg:ialm-stoc} and thus have a complete algorithm.  
Besides Assumptions~\ref{assump:oracle}-\ref{assump:reg-stoc}, we make the following assumptions.
\begin{assumption}[mean-squared smoothness]\label{assump:mean-smooth}
		For any $\vu, \vv \in \dom(h)$, $G_0(\cdot,\xi)$ and $\vC(\cdot,\xi)$ satisfy the mean-squared smoothness conditions: 
		\begin{equation*}
			\begin{aligned}
				&\EE_{\xi} \big[\Vert \nabla G_0(\vu,\xi) - \nabla G_0(\vv,\xi) \Vert^2\big] \le L_0^2 \Vert \vu-\vv \Vert^2, \\
				&\EE_{\xi} \big[\Vert J_\vC(\vu,\xi) - J_\vC(\vv,\xi) \Vert^2\big] \le L_J^2 \Vert \vu-\vv \Vert^2, \\ 
				&\EE_{\xi_1,\xi_2} \big[\Vert J_\vC(\vu,\xi_1)^{\top} \vC(\vu,\xi_2) - J_\vC(\vv,\xi_1)^{\top} \vC(\vv,\xi_2) \Vert^2\big]  \le L_J^2 \Vert \vu-\vv \Vert^2,
			\end{aligned}
		\end{equation*}
where $\xi_1$ and $\xi_2$ are independent and follow the same distribution as $\xi$ in \eqref{eq:ncp-stoc}.		
	\end{assumption}
\begin{remark}
Mean-squared smoothness is needed to have accelerated convergence for a stochastic gradient-type method on solving nonconvex stochastic problems \cite{fang2018spider, cutkosky2019momentum, tran2022hybrid, arjevani2022lower, xu2022momentum}. It naturally holds for the special case in 
\eqref{eq:ncp-finite} if each component of the objective and constraint functions is smooth. This condition is crucial to obtain our $O(\vareps^{-5})$ complexity result. However, the methods in \cite{ma2019proximally, boob2022stochastic}  can still only achieve a result of $O(\vareps^{-6})$ even with the mean-squared smoothness condition, as they do not exploit the structure.
\end{remark}
	
	\begin{assumption}[unbiasedness and bounded variance]\label{assump:var}
		For any $\vx\in\dom(h)$, the objective and constraint functions 
		satisfy 
		\begin{equation}\label{eq:unbias}
		\EE_{\xi} [\nabla G_0(\vx,\xi)] = \nabla g(\vx), \quad \EE_{\xi} [J_\vC(\vx,\xi)] = J_\vc(\vx).
		\end{equation}
		Also, there exist $\sigma_g, \sigma_c > 0$ such that for any $\vx\in\dom(h)$,
		\begin{equation*}
			\begin{aligned}
				&\EE_{\xi} \left[\Vert \nabla G_0(\vx,\xi)-\nabla g(\vx) \Vert^2\right] \le \sigma_g^2,\\
				& \EE_{\xi} \left[\Vert J_\vC(\vx,\xi) - J_\vc(\vx) \Vert_2^2\right] \le \sigma_c^2, \\ 
				&\EE_{\xi_1,\xi_2} \left[\Vert J_\vC(\vx,\xi_1)^{\top}\vC(\vx,\xi_2) - J_\vc(\vx)^{\top}\vc(\vx) \Vert^2\right] \le \sigma_c^2 , 
			\end{aligned}
		\end{equation*} 
where $\xi_1$ and $\xi_2$ are independent and follow the same distribution as $\xi$.		
	\end{assumption}

\begin{remark}\label{rem:smooth}
Under Assumption~\ref{assump:mean-smooth} and the unbiasedness condition \eqref{eq:unbias}, it can be easily shown that $g(\cdot)$ is $L_0$-smooth and $\vc(\cdot)$ is $L_J$-smooth; see the arguments at the end of section 2.2 of \cite{tran2022hybrid}. 
\end{remark}


\subsection{PStorm subroutine} \label{sec:pstorm}
The smooth part of the AL function $\cL_{\beta_k}(\cdot, \vy^k)$ has a smoothness parameter depending on $\beta_k$ that eventually depends on a given tolerance $\vareps$. Hence, to achieve a low-order overall complexity result, we need a subroutine whose complexity result has a low-order dependence not only on the pre-given stationarity violation but also on the smoothness parameter. With the mean-squared smoothness condition, the momentum-based variance-reduced proximal stochastic gradient method (PStorm) in \cite{xu2022momentum} is the one that meets our requirements. Below we first give a modified PStorm with a postprocessing step and then in the next subsection, discuss how to apply it to find $\vx^{k+1}$ in Algorithm~\ref{alg:ialm-stoc}.

Consider the problem
\begin{equation}\label{eq:comp-prob-stoc}
		F^*:= \min_{\vx\in\RR^d} ~\big\{F(\vx) := G(\vx)+H(\vx)\big\},
	\end{equation}
where $H$ is a closed convex function, and $G$ is smooth and possibly nonconvex.	 Let $\cA(\vx, \zeta)$ be a stochastic map that depends on a random variable $\zeta$. Suppose the following conditions hold: for some finite constants $L_G$ and $\sigma_G$,
\begin{subequations}\label{eq:cond-pstorm}
\begin{align}
&\EE_\zeta\big[\|\cA(\vx_1, \zeta)-\cA(\vx_2, \zeta)\|^2\big] \le L_G^2\|\vx_1-\vx_2\|^2,\forall\, \vx_1, \vx_2\in\dom(H), \\
&\EE_\zeta\big[\cA(\vx,\zeta)\big] = \nabla G(\vx), \ \EE\big[\|\cA(\vx,\zeta) - \nabla G(\vx)\|^2\big] \le \sigma_G^2, \forall\, \vx\in \dom(H).
\end{align}
\end{subequations} 
With $\cA(\cdot, \zeta)$ that satisfies the conditions above, we give the modified PStorm in Algorithm~\ref{alg:pstorm} and the complexity result in Lemma~\ref{lem:pstorm-thm3}.

	\begin{algorithm}[h] 
		\caption{
		PStorm$(G,H,\vx^0, L_G, \sigma_G, T, m_0, \bar{\eta}, \delta, \cA, \vareps)$ for solving \eqref{eq:comp-prob-stoc}}\label{alg:pstorm}
		\DontPrintSemicolon
		\textbf{Input:} initial point $\vx^0\in\dom(H)$, max iteration number $T$, smoothness constant $L_G$, variance bound $\sigma_G^2$, step size $\bar\eta$, momentum parameter $\delta \in (0,1)$, and an unbiased gradient estimator $\cA$ satisfying \eqref{eq:cond-pstorm}\;
		\textbf{Initialization:} Let $\vd^0 = \frac{1}{m_0} \sum_{\zeta \in B_0} \cA(\vx^0,\zeta)$ with $B_0$ containing $m_0$ i.i.d. samples.\;
		\For{$t=0,1,\ldots,T-1$}{
			\begin{equation}\label{eq:pstorm-x}
				\vx^{t+1} = \prox_{\bar{\eta} H}(\vx^{t} - \bar{\eta} \vd^t). 
			\end{equation}
			\hspace{-0.2cm}
			Compute
			\begin{align*}
				\vv^{t+1} &=  \cA(\vx^{t+1},\zeta^{t+1}), \quad 
				 \vu^{t+1} =  \cA(\vx^{t},\zeta^{t+1}). 
			\end{align*}
			\hspace{-0.2cm}
			Let $\vd^{t+1} = \vv^{t+1}+(1-\delta)(\vd^t-\vu^{t+1})$.
		}
		Choose $\vx^{\tau}$ uniformly at random from $\{\vx^0,\dots,\vx^{T-1}\}$. \;		
		\textbf{Sampling:} Set $m_1 = \left\lceil \frac{48\sigma_G^2}{\vareps^2} \right\rceil$, obtain a set $B_1$ of $m_1$ i.i.d samples of $\zeta$, and compute $\vv = \frac{1}{m_1}\sum_{\zeta \in B_{1}} \cA(\vx^{\tau},\zeta) $.\; 
		\textbf{Postprocessing:} output $\hat{\vx} = \prox_{\bar{\eta} H}(\vx^{\tau} - \bar{\eta} \vv)$. 
	\end{algorithm}

	\begin{lemma}\label{lem:pstorm-thm3} 
	Assume the conditions in \eqref{eq:cond-pstorm}. Given an error tolerance $\vareps>0$, choose parameters of Algorithm~\ref{alg:pstorm} as follows:
	\begin{equation}\label{eq:para-subroutine}
	\begin{aligned}
	&\bar\eta = \frac{\eta}{L_G \sqrt[3]{T}},\ \delta = \frac{4\eta^2+10\eta^2(2-\eta T^{-\frac{1}{3}})}{T^{\frac{2}{3}}+4\eta^2}, \ m_0 = \lceil c_0 \sqrt[3]{T}\rceil \\
	&T = \left\lceil \frac{ 48^{\frac{3}{2}} 40^{\frac{3}{2}}\left( \frac{L_G[F(\vx^0)-F^*]_{1+}}{\eta}+\frac{\sigma_G^2}{20c_0\eta^2}+\frac{24^2\sigma_G^2 \eta^2}{10} \right)^{\frac{3}{2}} }{\vareps^3} \right\rceil,
	\end{aligned}
	\end{equation}
for some positive constants $\eta$ and $c_0$, where $[a]_{1+}:=\max\{a,1\}$ for any $a\in\RR$.	If $\eta \le \frac{\sqrt[3]{T}}{10}$, then the output $\hat\vx$ satisfies $\EE\big[\dist(\vzero, \partial F(\hat{\vx}))^2\big] \le \vareps^2$.
	\end{lemma}
	
\begin{proof}
First,  directly from Corollary 2.2 of \cite{xu2022momentum}, we have  
		\begin{equation}\label{eq:pstorm-opt}
			\EE\left[\left\Vert \frac{1}{\bar{\eta}} \Big( \vx^\tau - \prox_{\bar{\eta} H}\big(\vx^{\tau}-\bar{\eta} \nabla G(\vx^\tau)\big)\Big)  \right\Vert^2\right] \le \frac{\vareps^2}{48}.
		\end{equation}		
		Hence, by the postprocessing step of Algorithm \ref{alg:pstorm}, it holds that 
		\begin{align*}
			\EE\left[\left\Vert \frac{\hat{\vx}-\vx^{\tau}}{\bar{\eta}} \right\Vert^2\right] &= \EE \left[\left\Vert \frac{1}{\bar{\eta}} \Big(\vx^\tau-\prox_{\bar{\eta}H}(\vx^\tau-\bar{\eta}\vv)\Big) \right\Vert^2\right] \nonumber \\
			&\le \EE \left[\left( \left\Vert \frac{1}{\bar{\eta}} \Big(\vx^\tau-\prox_{\bar{\eta}H}\big(\vx^\tau-\bar{\eta}\nabla G(\vx^\tau)\big)\Big) \right\Vert + \Vert \vv - \nabla G(\vx^\tau) \Vert \right)^2 \right] \nonumber \\
			&\le 2\EE \left[\left\Vert \frac{1}{\bar{\eta}} \Big(\vx^\tau-\prox_{\bar{\eta}H}\big(\vx^\tau-\bar{\eta}\nabla G(\vx^\tau)\big)\Big) \right\Vert^2\right] + 2\EE\big[\Vert \vv - \nabla G(\vx^\tau) \Vert^2\big], 
					\end{align*}
	where the first inequality follows from the 	nonexpansiveness of the proximal gradient mapping, and the second inequality holds due to the Young's inequality. Now by \eqref{eq:pstorm-opt} and noticing $\EE[\Vert \vv - \nabla G(\vx^\tau) \Vert^2 \le \frac{\sigma_G^2}{m_1} \le \frac{\vareps^2}{48}$, we have from the inequality above that
	\begin{equation} \label{eq:prox-bound}
	\EE\left[\left\Vert \frac{\hat{\vx}-\vx^{\tau}}{\bar{\eta}} \right\Vert^2\right] \le \frac{\vareps^2}{12}.
	\end{equation}
		
In addition, we have $\frac{\vx^\tau-\hat{\vx}}{\bar{\eta}} + \nabla G(\hat{\vx}) - \vv \in \partial F(\hat{\vx})$, and thus
		\begin{align*}
			\EE\big[\dist(\vzero, \partial F(\hat{\vx}))^2\big] &\le \EE\left[\left\Vert \frac{\vx^\tau-\hat{\vx}}{\bar{\eta}} + \nabla G(\hat{\vx}) - \nabla G(\vx^\tau) + \nabla G(\vx^\tau) - \vv \right\Vert^2\right] \nonumber\\
			&\le 3\EE \left[\left\Vert \frac{\vx^\tau-\hat{\vx}}{\bar{\eta}} \right\Vert^2\right] + 3\EE \big[\Vert \nabla G(\hat{\vx})-\nabla G(\vx^\tau) \Vert^2\big] + 3\EE \big[\Vert \nabla G(\vx^\tau) - \vv \Vert^2\big] \nonumber\\
			&\le (3+3L_G^2 \bar{\eta}^2) \EE\left[\left\Vert \frac{\hat{\vx}-\vx^\tau}{\bar{\eta}} \right\Vert^2\right] + \frac{3\vareps^2}{48} \nonumber\\
			&\le \frac{6\vareps^2}{12} + \frac{3\vareps^2}{6} \nonumber= \vareps^2,
		\end{align*}
		where we have used Young's inequality in the second inequality, the third inequality follows from the $L_G$-smoothness of $G$ and $\EE[\Vert \vv - \nabla G(\vx^\tau) \Vert^2 \le \frac{\vareps^2}{48}$, and the fourth inequality holds because of \eqref{eq:prox-bound} and $\bar{\eta} \le \frac{1}{L_G}$. This completes the proof. 
\qed\end{proof}	

Below we choose an appropriate $\eta$ and $c_0$ in Lemma~\ref{lem:pstorm-thm3} to obtain a complexity result that has a low-order dependence on $\sigma_G, L_G$ and $F(\vx^0)-F^*$.	

\begin{lemma}\label{lem:pstorm-comp} 
Assume the conditions in \eqref{eq:cond-pstorm}. Let $\vareps>0$ be given and satisfy $\vareps \le \frac{\sigma_G}{10}\sqrt{1920}\sqrt{3}\big(\frac{24^2}{10}\big)^{\frac{1}{2}}$. Choose parameters of Algorithm~\ref{alg:pstorm} as those in \eqref{eq:para-subroutine} with 
\begin{equation}\label{eq:choice-eta-c0}
\eta = \Big(\frac{10}{24^2}\Big)^{\frac{1}{3}}\frac{\big(L_G[F(\vx^0)-F^*]_{1+}\big)^{\frac{1}{3}}}{\sigma_G^{\frac{2}{3}}},\quad c_0 = \Big(\frac{24^2}{10}\Big)^{\frac{1}{3}}\frac{\sigma_G^{\frac{8}{3}}}{20 \big(L_G[F(\vx^0)-F^*]_{1+}\big)^{\frac{4}{3}}}.
\end{equation}
Then $\EE\big[\dist(\vzero, \partial F(\hat{\vx}))^2\big] \le \vareps^2$. The total number of calls to $\cA$ is
\begin{equation*}
\mathrm{Total}_\cA = \Theta\left(\frac{\sigma_G L_G[F(\vx^0)-F^*]_{1+}}{\vareps^3} + \frac{\sigma_G^3}{\vareps L_G[F(\vx^0)-F^*]_{1+}} + \frac{\sigma_G^{\frac{8}{3}}}{\big(L_G[F(\vx^0)-F^*]_{1+}\big)^{\frac{4}{3}}}\right),
\end{equation*}
where $[a]_{1+}:=\max\{a,1\}$ for any $a\in\RR$.
\end{lemma}

\begin{proof}
First, plugging the chosen $\eta$ and $c_0$ into \eqref{eq:para-subroutine}, we have
\begin{equation}\label{eq:form-T}
T = \left\lceil \frac{ 1920^{\frac{3}{2}} 3^{\frac{3}{2}} \left( \frac{24^2}{10} \right)^{\frac{1}{2}}\sigma_G L_G[F(\vx^0)-F^*]_{1+} }{\vareps^3} \right\rceil,
\end{equation}
and it is straightforward to verify $\eta \le \frac{\sqrt[3]{T}}{10}$ by the condition $\vareps \le \frac{\sigma_G}{10}\sqrt{1920}\sqrt{3}\big(\frac{24^2}{10}\big)^{\frac{1}{2}}$. Hence, from Lemma~\ref{lem:pstorm-thm3}, it follows that $\EE\big[\dist(\vzero, \partial F(\hat{\vx}))^2\big] \le \vareps^2$.

Second, notice that Algorithm~\ref{alg:pstorm} calls $\cA$ twice for each iteration. Hence, accounting the calls to $\cA$ in the initial and postprocessing steps, we obtain the total number of calls to $\cA$ is
\begin{align*}
&~2T + m_0 + m_1 = 2T + \left\lceil c_0 \sqrt[3]{T}\right\rceil + \left\lceil \frac{48\sigma_G^2}{\vareps^2} \right\rceil \le 2T + c_0 \sqrt[3]{T}+  \frac{48\sigma_G^2}{\vareps^2} + 2\\
\le &~ 2 \frac{ 1920^{\frac{3}{2}} 3^{\frac{3}{2}} \left( \frac{24^2}{10} \right)^{\frac{1}{2}}\sigma_G L_G[F(\vx^0)-F^*]_{1+} }{\vareps^3} + c_0 \left(\frac{ 1920^{\frac{3}{2}} 3^{\frac{3}{2}} \left( \frac{24^2}{10} \right)^{\frac{1}{2}}\sigma_G L_G[F(\vx^0)-F^*]_{1+} }{\vareps^3} \right)^{\frac{1}{3}}+  \frac{48\sigma_G^2}{\vareps^2} + c_0 + 4\\
\le &~ 2\cdot 1920^{\frac{3}{2}} 3^{\frac{3}{2}} \left( \frac{24^2}{10} \right)^{\frac{1}{2}}\left(\frac{\sigma_G L_G[F(\vx^0)-F^*]_{1+}}{\vareps^3} + \frac{\sigma_G^3}{\vareps L_G[F(\vx^0)-F^*]_{1+}} + \frac{\sigma_G^2}{\vareps^2} + \frac{\sigma_G^{\frac{8}{3}}}{\big(L_G[F(\vx^0)-F^*]_{1+}\big)^{\frac{4}{3}}}\right) + 4.
\end{align*}
Now notice $\frac{\sigma_G L_G[F(\vx^0)-F^*]_{1+}}{\vareps^3} + \frac{\sigma_G^3}{\vareps L_G[F(\vx^0)-F^*]_{1+}} \ge \frac{2\sigma_G^2}{\vareps^2}$ and absorb universal constants into $\Theta$. We obtain the desired result and complete the proof.
\qed\end{proof}

\begin{remark}\label{rm:pstorm-comp}
In the choice of $\eta$ and $c_0$ in Lemma~\ref{lem:pstorm-comp}, we have implicitly assumed $\sigma_G>0$. Hence, the claimed result does not apply to a deterministic scenario. Also, the setting of $\eta$ and $c_0$ in \eqref{eq:choice-eta-c0} needs the value of $F^*$ that is unknown. However, we can replace $F(\vx^0)-F^*$ by its upper bound, which can be easily obtained, as we will see for the subproblems of Algorithm~\ref{alg:ialm-stoc}.
\end{remark}

\subsection{Overall Complexity} \label{sec:ialm+pstorm}	
To apply Algorithm \ref{alg:pstorm} to find each $\vx^{k+1}$, the key is to build a stochastic map that satisfies conditions similar to those in \eqref{eq:cond-pstorm}. The following lemma gives the key.

\begin{lemma}\label{lem:delta}
Under Assumptions~\ref{assump:oracle}, \ref{assump:mean-smooth} and \ref{assump:var}, given $\vy\in\RR^m$, let $\Phi(\cdot):= \cL_{\beta}(\cdot,\vy)-h(\cdot)$ and 
\begin{equation}\label{eq:Delta}
	\Delta(\vx, \zeta) := \nabla G_0(\vx, \xi_1) + J_\vC(\vx,\xi_1)^\top \vy + \beta J_\vC(\vx,\xi_1)^\top \vC(\vx,\xi_2), 
	\end{equation}
where $\zeta=(\xi_1, \xi_2)$, and $\xi_1$ and $\xi_2$ are two independent random variables that follow the same distribution as $\xi$ in \eqref{eq:ncp-stoc}. Then it holds
\begin{subequations}\label{eq:cond-delta}
\begin{align}
&\EE_\zeta\big[\|\Delta(\vx_1, \zeta)-\Delta(\vx_2, \zeta)\|^2\big] \le L_\Phi^2\|\vx_1-\vx_2\|^2,\forall\, \vx_1, \vx_2\in\dom(h), \label{eq:cond-delta-1}\\
&\EE_\zeta\big[\Delta(\vx,\zeta)\big] = \nabla \Phi(\vx), \ \EE\big[\|\Delta(\vx,\zeta) - \nabla \Phi(\vx)\|^2\big] \le \sigma_\Phi^2, \forall\, \vx\in \dom(h), \label{eq:cond-delta-2}
\end{align}
\end{subequations}
where 
$$L_\Phi = \sqrt{3L_0^2 + 3L_J^2\|\vy\|^2 + 3\beta^2 L_J^2}, \quad \sigma_\Phi = \sqrt{3\sigma_g^2 + 3\sigma_c^2 \|\vy\|^2 + 3\beta^2 \sigma_c^2}.$$
\end{lemma}

\begin{proof}
Because $\xi_1$ and $\xi_2$ are independent, it holds
$$\EE_\zeta\big[\Delta(\vx, \zeta)\big] = \EE_{\xi_1}\big[\nabla G_0(\vx, \xi_1)\big] + \EE_{\xi_1}\big[J_\vC(\vx,\xi_1)^\top \vy\big] + \EE_{\xi_1}\big[ \beta J_\vC(\vx,\xi_1)\big]^\top \EE_{\xi_2}\big[\vC(\vx,\xi_2)\big].$$
Since $\xi_1$ and $\xi_2$ both follow the distribution of $\xi$, we have from the definition of $\vc(\cdot)$ in \eqref{eq:ncp-stoc} and \eqref{eq:unbias} that
$$\EE_\zeta\big[\Delta(\vx, \zeta)\big] = \nabla g(\vx) + J_\vc(\vx)^\top \vy + \beta J_\vc(\vx)^\top \vc(\vx) = \nabla \Phi(\vx).$$

In addition, by the Young's inequality, it follows that
\begin{align*}
&~\|\Delta(\vx_1, \zeta)-\Delta(\vx_2, \zeta)\|^2 \nonumber \\
\le & ~ 3\|\nabla G_0(\vx_1, \xi_1) - \nabla G_0(\vx_2, \xi_1)\|^2  + 3\|J_\vC(\vx_1,\xi_1)^\top \vy - J_\vC(\vx_2,\xi_1)^\top \vy\|^2 \nonumber \\
& ~ + 3\beta^2 \|J_\vC(\vx_1,\xi_1)^\top \vC(\vx_1,\xi_2) - J_\vC(\vx_2,\xi_1)^\top \vC(\vx_2,\xi_2)\|^2,
\end{align*}
which together with Assumption~\ref{assump:mean-smooth} gives
$$\EE_\zeta\big[\|\Delta(\vx_1, \zeta)-\Delta(\vx_2, \zeta)\|^2 \big] \le 3(L_0^2 + L_J^2\|\vy\|^2 + \beta^2 L_J^2)\|\vx_1-\vx_2\|^2.$$
Hence, \eqref{eq:cond-delta-1} holds. Similarly, by the Young's inequality and Assumption~\ref{assump:var}, we can show \eqref{eq:cond-delta-2} and complete the proof.
\qed\end{proof}
With Lemma~\ref{lem:delta}, we are able to apply Algorithm~\ref{alg:pstorm} to find each $\vx^{k+1}$. The theorem below gives the oracle complexity for the $k$-th outer iteration of Algorithm~\ref{alg:ialm-stoc}.

\begin{theorem}[Oracle complexity per outer iteration]\label{thm:k-th-complexity}
Let $(\vx^k, \vy^k)$ be the $k$-th iterate generated in Algorithm~\ref{alg:ialm-stoc} with a given tolerance $\vareps>0$. Define
$$F_k(\vx) := \cL_{\beta_k}(\vx, \vy^k),\quad F_k^*:=\min_\vx F_k(\vx), \quad \Phi_k (\vx) := F_k(\vx)  - h(\vx).$$ 
Under Assumptions~\ref{assump:oracle}, \ref{assump:mean-smooth} and \ref{assump:var}, if  $\vareps \le \frac{3\sqrt{\sigma_g^2 + \beta_0^2\sigma_c^2} }{10}\sqrt{1920}\big(\frac{24^2}{10}\big)^{\frac{1}{2}}$, then we can find $\vx^{k+1}$ that satisfies \eqref{eq:ialm-x-stoc} by Algorithm~\ref{alg:pstorm} with the following call
\begin{equation}\label{eq:k-call-pstorm}
\vx^{k+1} \gets \mathrm{PStorm}(\Phi_k, h, \vx^k, L_{\Phi_k}, \sigma_{\Phi_k}, T_k, m_{0,k}, \bar\eta_k, \delta_k, \Delta_k, \vareps),
\end{equation}
where  
\begin{align}\label{eq:L-sigma-phi-k}
&L_{\Phi_k} = \sqrt{3L_0^2 + 3L_J^2\|\vy^k\|^2 + 3\beta_k^2 L_J^2}, \quad \sigma_{\Phi_k} = \sqrt{3\sigma_g^2 + 3\sigma_c^2 \|\vy^k\|^2 + 3\beta_k^2 \sigma_c^2},\\
&\Delta_k(\vx, \zeta) := \nabla G_0(\vx, \xi_1) + J_\vC(\vx,\xi_1)^\top \vy^k + \beta_k J_\vC(\vx,\xi_1)^\top \vC(\vx,\xi_2). \label{eq:k-th-func}\\
&\bar\eta_k = \frac{\eta_k}{L_{\Phi_k} \sqrt[3]{T_k}},\ \delta_k = \frac{4\eta_k^2+10\eta_k^2(2-\eta_k T_k^{-\frac{1}{3}})}{T_k^{\frac{2}{3}}+4\eta_k^2}, \ m_{0,k} = \lceil c_{0,k} \sqrt[3]{T_k}\rceil \\
&T_k = \left\lceil \frac{ 48^{\frac{3}{2}} 40^{\frac{3}{2}}\left( \frac{L_{\Phi_k}[F_k(\vx^k)-F_k^*]_{1+}}{\eta}+\frac{\sigma_{\Phi_k}^2}{20c_{0,k}\eta_k^2}+\frac{24^2\sigma_{\Phi_k}^2 \eta_k^2}{10} \right)^{\frac{3}{2}} }{\vareps^3} \right\rceil
\end{align}
with
\begin{align}\label{eq:eta-k-c0k}
&\eta_k = \Big(\frac{10}{24^2}\Big)^{\frac{1}{3}}\frac{\big(L_{\Phi_k}[F_k(\vx^k)-F_k^*]_{1+}\big)^{\frac{1}{3}}}{\sigma_{\Phi_k}^{\frac{2}{3}}},\quad c_{0,k} = \Big(\frac{24^2}{10}\Big)^{\frac{1}{3}}\frac{\sigma_{\Phi_k}^{\frac{8}{3}}}{20 \big(L_{\Phi_k}[F_k(\vx^k)-F_k^*]_{1+}\big)^{\frac{4}{3}}}.
\end{align}
In addition, $\cT_k$ calls to the stochastic first-order oracle that is defined in Assumption~\ref{assump:oracle} will be enough to produce $\vx^{k+1}$, where
\begin{equation}\label{eq:cT-k}
\cT_k = O\left(\frac{\sigma_{\Phi_k} L_{\Phi_k}[F_k(\vx^k)-F_k^*]_{1+}}{\vareps^3} + \frac{\sigma_{\Phi_k}^3}{\vareps L_{\Phi_k}} + \frac{\sigma_{\Phi_k}^{\frac{8}{3}}}{L_{\Phi_k}^{\frac{4}{3}}}\right).
\end{equation}
\end{theorem}

\begin{proof}
By Lemma~\ref{lem:delta}, it holds
\begin{subequations}\label{eq:cond-delta-k}
\begin{align}
&\EE_\zeta\big[\|\Delta_k(\vx_1, \zeta)-\Delta_k(\vx_2, \zeta)\|^2\big] \le L_{\Phi_k}^2\|\vx_1-\vx_2\|^2,\forall\, \vx_1, \vx_2\in\dom(h), \label{eq:cond-delta-k-1}\\
&\EE_\zeta\big[\Delta_k(\vx,\zeta)\big] = \nabla \Phi_k(\vx), \ \EE\big[\|\Delta_k(\vx,\zeta) - \nabla \Phi_k(\vx)\|^2\big] \le \sigma_{\Phi_k}^2, \forall\, \vx\in \dom(h), \label{eq:cond-delta-k-2}
\end{align}
\end{subequations}
where $L_{\Phi_k}$ and $\sigma_{\Phi_k}$ are given in \eqref{eq:L-sigma-phi-k}. 
Hence, $\vareps \le \frac{\sigma_{\Phi_k}}{10}\sqrt{1920}\sqrt{3}\big(\frac{24^2}{10}\big)^{\frac{1}{2}}$ by $\beta_0\le \beta_k,\forall\, k\ge0$ and the assumed condition on $\vareps$. Thus, from Lemma~\ref{lem:pstorm-comp}, the point $\vx^{k+1}$ returned by Algorithm~\ref{alg:pstorm} with the call in \eqref{eq:k-call-pstorm} is an $\vareps$-stationary point $\vx^{k+1}$ of $F_k(\cdot)$ in expectation, i.e., $\EE \big[\dist(\vzero, \partial_x \cL_{\beta_k}(\vx^{k+1},\vy^k))^2 \, |\, \vy^k\big] \le \vareps^2$ and, the total number of calls to $\Delta_k$ is
\begin{align}\label{eq:oracle-Delta-k}
\mathrm{Total}_{\Delta_k} =  & ~\Theta\left(\frac{\sigma_{\Phi_k} L_{\Phi_k}[F_k(\vx^k)-F_k^*]_{1+}}{\vareps^3} + \frac{\sigma_{\Phi_k}^3}{\vareps L_{\Phi_k}[F_k(\vx^k)-F_k^*]_{1+}} + \frac{\sigma_{\Phi_k}^{\frac{8}{3}}}{\big(L_{\Phi_k}[F_k(\vx^k)-F_k^*]_{1+}\big)^{\frac{4}{3}}}\right)\nonumber \\
= & ~ O\left(\frac{\sigma_{\Phi_k} L_{\Phi_k}[F_k(\vx^k)-F_k^*]_{1+}}{\vareps^3} + \frac{\sigma_{\Phi_k}^3}{\vareps L_{\Phi_k}} + \frac{\sigma_{\Phi_k}^{\frac{8}{3}}}{L_{\Phi_k}^{\frac{4}{3}}}\right),
\end{align}
where we have used $[F_k(\vx^k)-F_k^*]_{1+} \ge 1$ in the second equation. Since each call to $\Delta_k$ will need two stochastic first-order oracles as we assumed in Assumption~\ref{assump:oracle}, we have $\cT_k = 2\cdot \mathrm{Total}_{\Delta_k}$ and thus complete the proof.
\qed\end{proof}

\begin{remark}
In the parameter settings of \eqref{eq:L-sigma-phi-k}-\eqref{eq:eta-k-c0k}, we used the unknown value $F_k^*$. As we point out in Remark~\ref{rm:pstorm-comp}, we can replace $F_k(\vx^k)-F_k^*$ by its upper bound such as the one we establish in \eqref{eq:bd-Fk-Fk} below.
\end{remark}

The lemma below is used to bound $\EE[F_k(\vx^k) - F_k^*]$.

\begin{lemma}\label{lem:bd-Fk-Fbar}
Let $F_k(\cdot)$ and $F_k^*$ be defined in Theorem~\ref{thm:k-th-complexity}. If \eqref{eq:y-bound-stoc} and  \eqref{eq:c-bound} hold, then
\begin{align}\label{eq:bd-Fk-Fk-E-k0}
F_0(\vx^0) - F_0^* \le 2 B_0 + \frac{\beta_0}{2}\|\vc(\vx^0)\|^2,
\end{align}
and for any $k\ge1$,
\begin{equation}\label{eq:bd-Fk-Fk-E-gk}
\begin{aligned}
&~\EE[F_k(\vx^k) - F_k^*] \\
\le & ~ B_F := 2 B_0 + \frac{\vareps^2 + B_0^2 + M^2}{v\beta_0}\Big(\frac{2\sigma}{v}+1\Big) +  + \frac{\gamma_0^2}{2v\beta_0}\Big(\frac{2\sigma}{v}+1\Big)\big(v+2B_c^2\big)\frac{4}{(\ln\sigma)^2\sigma^{\frac{2}{\ln \sigma}}}.
\end{aligned}
\end{equation}
\end{lemma}

\begin{proof}
Suppose $\hat\vx^k = \argmin_\vx F_k(\vx)$, i.e., $F_k^* = F_k(\hat\vx^k)$. Then it holds
\begin{align}
F_k(\vx^k) - F_k^* = & ~f_0(\vx^k)  + \langle \vy^k, \vc(\vx^k) \rangle + \frac{\beta_k}{2}\|\vc(\vx^k)\|^2 - f_0(\hat\vx^k)  - \langle \vy^k, \vc(\hat \vx^k) \rangle - \frac{\beta_k}{2}\|\vc(\hat\vx^k)\|^2 \nonumber \\
\le & ~ f_0(\vx^k)  + \langle \vy^k, \vc(\vx^k) \rangle + \frac{\beta_k}{2}\|\vc(\vx^k)\|^2 - f_0(\hat\vx^k) + \frac{1}{2\beta_k}\|\vy^k\|^2 \nonumber\\
\le & ~ 2 B_0  + k\gamma_0 \|\vc(\vx^k)\| + \frac{\beta_k}{2}\|\vc(\vx^k)\|^2 + \frac{k^2\gamma_0^2}{2\beta_k}, \forall\, k \ge 0, \label{eq:bd-Fk-Fk}
\end{align}
where the first inequality is by the Young's inequality, and the second inequality follows from \eqref{P3-eqn-33-stoc-a}. Hence for $k=0$, taking expectation on both sides of \eqref{eq:bd-Fk-Fk} gives \eqref{eq:bd-Fk-Fk-E-k0},
and for $k\ge1$,
\begin{align}\label{eq:bd-Fk-Fk-E}
\EE[F_k(\vx^k) - F_k^*] \le & ~ 2 B_0 + k\gamma_0 \EE\big[\|\vc(\vx^k)\|\big] + \frac{\beta_k}{2}\EE\big[\|\vc(\vx^k)\|^2\big] + \frac{k^2\gamma_0^2}{2\beta_k} \nonumber \\
\le & ~ 2 B_0 + \frac{2k\gamma_0\sqrt{\vareps^2 + B_0^2 + B_c^2 (k-1)^2\gamma_0^2 + M^2}}{v\beta_{k-1}} \nonumber \\
& ~ + \frac{2\beta_k}{v^2 \beta_{k-1}^2} \left(\vareps^2 + B_0^2 + B_c^2 (k-1)^2\gamma_0^2 + M^2\right) + \frac{k^2\gamma_0^2}{2\beta_k} \nonumber \\
\le & ~ 2 B_0 + \frac{k^2\gamma_0^2}{v\beta_{k-1}}  + \left(\frac{2\beta_k}{v^2 \beta_{k-1}^2} + \frac{1}{v\beta_{k-1}} \right)(\vareps^2 + B_0^2 + B_c^2 (k-1)^2\gamma_0^2 + M^2) + \frac{k^2\gamma_0^2}{2\beta_k}, 
\end{align}
where in the second inequality we have used \eqref{eq:y-bound-stoc}, \eqref{eq:c-bound} and the Jensen's inequality, and the third inequality follows from the Young's inequality.

Notice that $\frac{x^2}{\sigma^x}$ attains its maximum at $x = \frac{2}{\ln \sigma}$ for $x>0$. Thus 
\begin{equation}\label{eq:max-ratio}
\frac{x^2}{\sigma^x} \le \frac{4}{(\ln\sigma)^2\sigma^{\frac{2}{\ln \sigma}}},\forall\, x >0.
\end{equation}
Now by $\beta_k = \beta_0\sigma^k$, we have for any $k\ge1$, 
\begin{align*}
&\frac{2\beta_k}{v^2 \beta_{k-1}^2} + \frac{1}{v\beta_{k-1}} = \frac{2\sigma}{v^2 \beta_0 \sigma^{k-1}} + \frac{1}{v\beta_0\sigma^{k-1}} \le \frac{2\sigma}{v^2 \beta_0} + \frac{1}{v\beta_0},\\
&\frac{k^2\gamma_0^2}{v\beta_{k-1}} + \frac{k^2\gamma_0^2}{2\beta_k} = \Big(\frac{\sigma\gamma_0^2}{v\beta_0} + \frac{\gamma_0^2}{2\beta_0}\Big) \frac{k^2}{\sigma^k} \le \Big(\frac{\sigma\gamma_0^2}{v\beta_0} + \frac{\gamma_0^2}{2\beta_0}\Big) \frac{4}{(\ln\sigma)^2\sigma^{\frac{2}{\ln \sigma}}},\\
&\left(\frac{2\beta_k}{v^2 \beta_{k-1}^2} + \frac{1}{v\beta_{k-1}} \right) B_c^2 (k-1)^2\gamma_0^2 \le B_c^2\gamma_0^2\Big(\frac{2\sigma}{v^2 \beta_0} + \frac{1}{v\beta_0}\Big) \frac{4}{(\ln\sigma)^2\sigma^{\frac{2}{\ln \sigma}}}.
\end{align*}
Plugging the three inequalities above into \eqref{eq:bd-Fk-Fk-E} and using the definition of $B_F$, we obtain \eqref{eq:bd-Fk-Fk-E-gk}.
\qed\end{proof}

By Theorem~\ref{thm:k-th-complexity} and Lemma~\ref{lem:bd-Fk-Fbar}, we are ready to show the overall oracle complexity of Algorithm~\ref{alg:ialm-stoc} to produce an $\vareps$-KKT point of \eqref{eq:ncp-stoc} in expectation.

	\begin{theorem}[Overall complexity of Algorithm~\ref{alg:ialm-stoc}]\label{thm:ialm-compl-exp-1}
		Under Assumptions~\ref{assump:oracle} through \ref{assump:var}, let $\vareps\in(0,1)$ be a given tolerance. Suppose $\vareps \le \frac{3\sqrt{\sigma_g^2 + \beta_0^2\sigma_c^2} }{10}\sqrt{1920}\big(\frac{24^2}{10}\big)^{\frac{1}{2}}$. In Algorithm~\ref{alg:ialm-stoc}, set $M_k=\Theta(\vareps^{-2})$ and $\gamma_k = \gamma_0,\forall\, k\ge0$ for some $\gamma_0 > 0$ such that  
$\frac{\sqrt{8} B_c \gamma_0}{\beta_0 v \vareps} \ge \frac{8}{\ln\sigma}$. Then it can produce an $\vareps$-KKT point of \eqref{eq:ncp-stoc} in expectation, by using Algorithm~\ref{alg:pstorm} as the subroutine and calling it via \eqref{eq:k-call-pstorm}. In addition, the total number $\mathrm{Oracle}_{\mathrm{total}}$ of calls to the oracle defined in Assumption~\ref{assump:oracle} satisfies 
		$$\EE\big[\mathrm{Oracle}_{\mathrm{total}}\big] = O(\vareps^{-5}).$$
	\end{theorem} 

\begin{proof}
From Theorem~\ref{thm:stoc-ialm-conv}, we know that $\vx^K$ is  an $\vareps$-KKT point of \eqref{eq:ncp-stoc} in expectation, where $K$ is given in \eqref{eq:K-stoc}.
Hence, the oracle complexity $\mathrm{Oracle}_{\mathrm{total}}$ of Algorithm~\ref{alg:ialm-stoc} is upper bounded by $\sum_{k=0}^{K-1}(\cT_k + M_k)$ with $\cT_k$ defined in \eqref{eq:cT-k}. To upper bound $\EE\big[\mathrm{Oracle}_{\mathrm{total}}\big]$, it suffices to upper bound $\EE\big[\cT_k\big]$ for each $k < K$.

By $\beta_k = \beta_0\sigma^k$ and \eqref{eq:y-bound-stoc}, we have from \eqref{eq:L-sigma-phi-k} that 
\begin{equation}\label{eq:boud-L-Phi-k}
\sqrt{3}\beta_0 L_J \sigma^{k} \le L_{\Phi_k} \le \sqrt{3L_0^2 + 3L_J^2 k^2 \gamma_0^2 + 3\beta_0^2 L_J^2 \sigma^{2k}} \le \sqrt{3}\big(L_0 + L_J k\gamma_0 + \beta_0 L_J\sigma^k\big).
\end{equation}
and
\begin{equation}\label{eq:boud-sigma-Phi-k}
\sigma_{\Phi_k} \le \sqrt{3\sigma_g^2 + 3\sigma_c^2 k^2 \gamma_0^2 + 3\beta_0^2 \sigma_c^2 \sigma^{2k}} \le \sqrt{3}\big(\sigma_g + \sigma_c k\gamma_0 + \beta_0\sigma_c \sigma^k\big).
\end{equation}
Hence,
\begin{align}\label{eq:bd-term2}
&~\frac{\sigma_{\Phi_k}^3}{\vareps L_{\Phi_k}} + \frac{\sigma_{\Phi_k}^{\frac{8}{3}}}{L_{\Phi_k}^{\frac{4}{3}}} = \frac{\sigma_{\Phi_k}}{\vareps}\frac{\sigma_{\Phi_k}^2}{L_{\Phi_k}} + \Big(\frac{\sigma_{\Phi_k}^{2}}{L_{\Phi_k}} \Big)^{\frac{4}{3}} \nonumber \\
\le & ~  \frac{\sqrt{3}\big(\sigma_g + \sigma_c k\gamma_0 + \beta_0\sigma_c \sigma^k\big)}{\vareps}\cdot\frac{3\sigma_g^2 + 3\sigma_c^2 k^2 \gamma_0^2 + 3\beta_0^2 \sigma_c^2 \sigma^{2k}}{\sqrt{3}\beta_0 L_J \sigma^{k}} + \Big(\frac{3\sigma_g^2 + 3\sigma_c^2 k^2 \gamma_0^2 + 3\beta_0^2 \sigma_c^2 \sigma^{2k}}{\sqrt{3}\beta_0 L_J \sigma^{k}}\Big)^{\frac{4}{3}}. 
\end{align}
By \eqref{eq:max-ratio}, it holds $\frac{3\sigma_c^2 k^2 \gamma_0^2}{\sqrt{3}\beta_0 L_J \sigma^{k}} \le \frac{3\sigma_c^2 \gamma_0^2}{\sqrt{3}\beta_0 L_J} \frac{4}{(\ln\sigma)^2\sigma^{\frac{2}{\ln \sigma}}}$ and $k\le \sigma^k \frac{4}{(\ln\sigma)^2\sigma^{\frac{2}{\ln \sigma}}}$. Hence, 
\begin{equation}\label{eq:order-sigma-k}
\sqrt{3}\big(\sigma_g + \sigma_c k\gamma_0 + \beta_0\sigma_c \sigma^k\big) = \Theta(\sigma^k),\quad \frac{3\sigma_g^2 + 3\sigma_c^2 k^2 \gamma_0^2 + 3\beta_0^2 \sigma_c^2 \sigma^{2k}}{\sqrt{3}\beta_0 L_J \sigma^{k}} = \Theta(\sigma^k).
\end{equation}
Therefore, \eqref{eq:bd-term2} implies
\begin{align}\label{eq:bd-term2-2}
&~\sum_{k=0}^{K-1}\left(\frac{\sigma_{\Phi_k}^3}{\vareps L_{\Phi_k}} + \frac{\sigma_{\Phi_k}^{\frac{8}{3}}}{L_{\Phi_k}^{\frac{4}{3}}} \right) = O\left(\sum_{k=0}^{K-1} \Big(\frac{\sigma^{2k}}{\vareps} + \sigma^{\frac{4k}{3}}\Big)\right) = O\left(\frac{\sigma^{2K}}{\vareps}\right)
\end{align}

In addition, $\EE\left[\sigma_{\Phi_k} L_{\Phi_k}[F_k(\vx^k)-F_k^*]_{1+} \right] \le \EE\left[\sigma_{\Phi_k} L_{\Phi_k}(F_k(\vx^k)-F_k^*) \right] + \sigma_{\Phi_k} L_{\Phi_k}$. Hence, from Lemma~\ref {lem:bd-Fk-Fbar}, \eqref{eq:boud-L-Phi-k}, and \eqref{eq:boud-sigma-Phi-k}, we obtain
\begin{align*}
&~\EE\left[\sigma_{\Phi_k} L_{\Phi_k}[F_k(\vx^k)-F_k^*]_{1+} \right] \\
\le & ~ 3\big(L_0 + L_J k\gamma_0 + \beta_0 L_J\sigma^k\big)\big(\sigma_g + \sigma_c k\gamma_0 + \beta_0\sigma_c \sigma^k\big)\left(1+\max\left\{B_F, 2 B_0 + \frac{\beta_0}{2}\|\vc(\vx^0)\|^2\right\}\right),
\end{align*}
which together with \eqref{eq:order-sigma-k} gives
\begin{align}\label{eq:bd-term1}
\sum_{k=0}^{K-1}\EE\left[\sigma_{\Phi_k} L_{\Phi_k}[F_k(\vx^k)-F_k^*]_{1+} \right] = O\left(\sum_{k=0}^{K-1}\sigma^{2k}\right) = O(\sigma^{2K}).
\end{align}

Therefore, from \eqref{eq:cT-k}, \eqref{eq:bd-term2-2}, and \eqref{eq:bd-term1}, it follows that $\EE\big[\sum_{k=0}^{K-1}\cT_k\big] = O\left(\frac{\sigma^{2K}}{\vareps^3}\right)$. Plugging $K$ given in \eqref{eq:K-stoc}, we have $\EE\big[\sum_{k=0}^{K-1}\cT_k\big] = O(\vareps^{-5})$. Since $M_k=O(\vareps^{-2})$ for all $k$, we obtain $\EE\big[\mathrm{Oracle}_{\mathrm{total}}\big]=\EE\big[\sum_{k=0}^{K-1}(\cT_k + M_k)\big]=O(\vareps^{-5})$ and complete the proof.
\qed\end{proof}


\section{Numerical Results}
In this section, we demonstrate the numerical performance of the proposed Stoc-iALM in Algorithm \ref{alg:ialm-stoc} (with PStorm in Algorithm~\ref{alg:pstorm} as the subroutine) on solving a fairness constrained problem and a Neyman-Pearson Classification problem. We compare it to the IPC method in \cite{ma2019proximally} that achieves the state-of-the-art complexity for solving \eqref{eq:ncp-stoc}.   
All the tests were performed in MATLAB 2019b on a Macbook Pro with 4 cores and 16GB memory. 


\subsection{Nonconvex Fairness Constrained Problem}\label{subsec:num_fair} 
Let $\vx$ denote the parameters of a linear model and $f(\vx; \va,b) = \phi_{\alpha}(l(\vx; \va,b))$ be the truncated logistic loss function, where $l(\vx;\va,b) = \log (1 + \exp (-b\va^{\top}\vx)), \phi_{\alpha}(s) = \alpha \log (1+\frac{s}{\alpha})$, and $\alpha = 2$ is set in our tests. Suppose there is a labeled dataset $D = \{(\va_i,b_i)\}_{i=1}^{|D|}$, a possibly unlabeled dataset $S = \{\va_j\}_{j=1}^{|S|}$, and a subset $S_{\min} \subseteq S$ of the minority population in $S$. Then the problem of training $\vx$ using the loss $f(\vx; \va,b)$ with a fairness constraint \cite{ma2019proximally} can be formulated as
\begin{equation}\label{eq:fair}
	\begin{aligned}
	\min_{\vx \in \RR^d} & \ f_0(\vx):= \frac{1}{|D|} \sum_{(\va,b) \in D} f(\vx;\va,b) , \\ 
	\st & \ f_1(\vx):= c\sum_{\va \in S} \sigma(\va^{\top}\vx) - \sum_{\va \in S_{\min}} \sigma(\va^{\top}\vx) \le 0,
	\end{aligned}
\end{equation}
	where $c \in (0,1)$ is a fairness parameter and $\sigma(s) = \frac{\exp(s)}{1+\exp(s)}$. The fairness constraint above aims at forcing the classifier to have a positive prediction on the minority group often enough.
Following \cite{ma2019proximally}, we use three data sets: \emph{bank-marketing}  from UCI repository \cite{Dua:2019} (shortened as \emph{bank} below) with $d = 81$ and $(|D|,|S|,|S_{\min}|) = (22605, 22605, 233)$, \emph{a9a} from LIBSVM library \cite{chang2011libsvm} with $d = 123$ and $(|D|,|S|,|S_{\min}|) = (32561, 16281, 1561)$, and \emph{loan} from LendingClub (which contains the information of $128375$ loans issued in the fourth quarter of 2018; see \cite{ma2019proximally} for more description) with $d = 250$ and $(|D|,|S|,|S_{\min}|) = (64485, 63890, 31966)$.  
We set the fairness parameter to $c = 0.4$ for the \emph{bank} dataset, $c = 0.1$ for the \emph{a9a} dataset, and $c = 0.6$ for the \emph{loan} dataset. 
	
To solve \eqref{eq:fair} by 
the proposed Stoc-iALM,  
we reformulate its inequality constraint to an equality constraint $f_1(\vx) + \vv = 0$ where $\vv \ge 0$ is enforced. Notice that the reformulation has equivalent stationarity conditions to the original model \eqref{eq:fair} as shown in \cite{li2021rate}. 
The IPC method in \cite{ma2019proximally} is applied directly to \eqref{eq:fair}, and following \cite{ma2019proximally}, we adopt its deterministic version. 
We only compare to the IPC, as it is demonstrated in \cite{ma2019proximally} to outperform other methods on solving \eqref{eq:fair} such as the Penalty with trust region method in \cite{cartis2011evaluation} and the subgradient method in \cite{yu2017online}. 

The tolerance is set to $\vareps = 0.01$ in all tests. Our proposed Stoc-iALM is terminated if both primal and dual residuals  
(see Definition \ref{def:eps-kkt-stoc}) 
of the equality-constrained reformulation of \eqref{eq:fair} are below $\vareps$. The IPC method  does not generate a dual iterate, so for a fair comparison, we compute an optimal dual variable $\vz \ge 0$ that minimizes the squared sum of the violation to the dual feasibility and the complementary slackness conditions 
of \eqref{eq:fair}:
\begin{equation}\label{eq:opt-dual}
\min_{\vz \ge 0} \|\nabla f_0(\vx) + J_{f_1}(\vx)^\top  \vz\|^2 + |\vz^\top f_1(\vx)|^2.
\end{equation}
Since $f_1$ is a scalar function in \eqref{eq:fair}, it is not difficult to have the optimal $\vz = \left[ -\frac{\nabla f_0(\vx)^\top \nabla f_1(\vx)}{f_1(\vx)^2 + \Vert \nabla f_1(\vx) \Vert^2} \right]_+$. 
Given this $\vz$, the IPC is terminated if both primal residual $\Vert [f_1(\vx)]_+ \Vert$ and dual residual $\|\nabla f_0(\vx) + \vz \nabla f_1(\vx) \|$ are below the given tolerance $\vareps$. 
For Stoc-iALM, at the $k$-th outer iteration, we set 
$\beta_k = 2.5^k$ and the smoothness parameter to $10 + \beta_k$, and we set $\tilde\vc(\vx^{k+1}) = \vc(\vx^{k+1})$ in the $\vy$-update \eqref{eq:alm-y-stoc}. In the PStorm subroutine, we set the mini-batch size to $30$ for all three data sets. The parameter settings of the IPC exactly follow from the code of \cite{ma2019proximally} that was kindly provided by the authors. The primal and dual residuals are recorded after every 50 inner iterations for Stoc-iALM, and after every data pass for IPC. Both methods start from a zero vector. As the proposed method is randomized, we run it for 10 independent trials by using different random seeds, while IPC is deterministic and thus we only perform one trial. 

Table \ref{table:fair_10seeds} lists the violation of primal feasibility and the violation to the dual feasibility at the produced $\vareps$-KKT point, and the number of data passes (shortened by \verb|pres|, \verb|dres| and \verb|#data| respectively) for each method to produce such a point.
Figure \ref{fig:fair_10seeds} plots the curves of the constraint function value and \verb|dres| at generated iterates by the proposed Stoc-iALM and the IPC, where the solid red curve shows the average results and the shadow area represents the standard deviation for the proposed method. To clearly show the difference of the results by the proposed method and IPC, we only plot the curves by IPC up to 20 number of data passes for the \emph{bank} and \emph{a9a} data sets. From the results in Table \ref{table:fair_10seeds} and Figure \ref{fig:fair_10seeds}, we see that both methods can reduce \verb|pres| and \verb|dres| below the given tolerance. However, the IPC needs significantly more data passes, especially for the \emph{bank} and \emph{a9a} data sets. In addition, the proposed method can perform well for all 10 trials.

\begin{table}\caption{The violation of primal feasibility and the violation to the dual feasibility at the produced $\vareps$-KKT point with $\vareps=10^{-2}$, and the number of data passes (shortened by pres, dres and \#data$^\dagger$ respectively) of 10 trials with different random seeds by the proposed Stoc-iALM and the IPC in \cite{ma2019proximally} on solving the fairness constrained problem \eqref{eq:fair} with \emph{a9a}, \emph{bank}, and \emph{loan} data sets (from left to right).\\[0.2cm]
$^\dagger$ \#data by Stoc-iALM are fractional because the subroutine PStorm uses minibatch of data points to compute sample gradients and we record pres and dres after every 50 inner iterations.}\label{table:fair_10seeds} 
\begin{center}
\resizebox{1 \textwidth}{!}{ 
\begin{tabular}{|c||c|c|c|} 
\hline
method [\emph{a9a}] & pres & dres & \#data \\\hline\hline 
Stoc-iALM (1) & 5.7e-3 & 9.3e-3 & 3.83  \\ 
Stoc-iALM (2) & 5.8e-3 & 9.5e-3 & 3.90  \\ 
Stoc-iALM (3) & 5.7e-3 & 9.4e-3 & 3.97  \\ 
Stoc-iALM (4) & 6.0e-3 & 1.00e-2 & 3.63  \\ 
Stoc-iALM (5) & 5.8e-3 & 9.2e-3 & 3.83  \\ 
Stoc-iALM (6) & 5.7e-3 & 9.9e-3 & 3.56  \\ 
Stoc-iALM (7) & 5.6e-3 & 9.5e-3 & 3.69  \\ 
Stoc-iALM (8) & 5.7e-3 & 9.3e-3 & 3.90  \\ 
Stoc-iALM (9) & 5.5e-3 & 8.6e-3 & 4.11  \\ 
Stoc-iALM (10) & 5.7e-3 & 9.9e-3 & 4.46  \\ 
\hline
IPC & 0 & 9.1e-3 & 111  \\\hline 
\end{tabular}

\begin{tabular}{|c||c|c|c|} 
\hline
method [\emph{bank}] & pres & dres & \#data \\\hline\hline 
Stoc-iALM (1) & 4.1e-3 & 9.7e-3 & 4.78  \\ 
Stoc-iALM (2) & 4.1e-3 & 9.9e-3 & 4.72  \\ 
Stoc-iALM (3) & 4.2e-3 & 1.00e-2 & 4.72  \\ 
Stoc-iALM (4) & 4.2e-3 & 9.9e-3 & 4.72  \\ 
Stoc-iALM (5) & 4.1e-3 & 9.5e-3 & 4.85  \\ 
Stoc-iALM (6) & 4.1e-3 & 1.00e-2 & 4.72  \\ 
Stoc-iALM (7) & 4.1e-3 & 9.8e-3 & 4.72  \\ 
Stoc-iALM (8) & 4.1e-3 & 9.9e-3 & 4.72  \\ 
Stoc-iALM (9) & 4.0e-3 & 9.9e-3 & 4.72  \\ 
Stoc-iALM (10) & 4.1e-3 & 9.8e-3 & 4.72  \\ 
\hline
IPC & 0 & 9.7e-3 & 107  \\\hline 
\end{tabular}

\begin{tabular}{|c||c|c|c|} 
\hline
method [\emph{loan}] & pres & dres & \#data \\\hline\hline 
Stoc-iALM (1) & 2.2e-4 & 9.1e-3 & 1.21  \\ 
Stoc-iALM (2) & 2.6e-4 & 9.7e-3 & 1.23  \\ 
Stoc-iALM (3) & 2.3e-4 & 9.3e-3 & 1.21  \\ 
Stoc-iALM (4) & 2.4e-4 & 9.6e-3 & 1.21  \\ 
Stoc-iALM (5) & 1.9e-4 & 8.3e-3 & 1.21  \\ 
Stoc-iALM (6) & 2.5e-4 & 9.2e-3 & 1.23  \\ 
Stoc-iALM (7) & 2.1e-4 & 9.0e-3 & 1.21  \\ 
Stoc-iALM (8) & 2.6e-4 & 9.8e-3 & 1.21  \\ 
Stoc-iALM (9) & 2.0e-4 & 8.8e-3 & 1.23  \\ 
Stoc-iALM (10) & 2.5e-4 & 1.00e-2 & 1.21  \\ 
\hline
IPC & 0 & 2.8e-3 & 3  \\\hline 
\end{tabular}
}
\end{center}
\end{table}

\begin{figure}[t] 
	\begin{center}
		\begin{tabular}{cc} 
\multicolumn{2}{c}{{\Large\emph{bank} dataset}}  \\
    \includegraphics[width=0.35\textwidth]{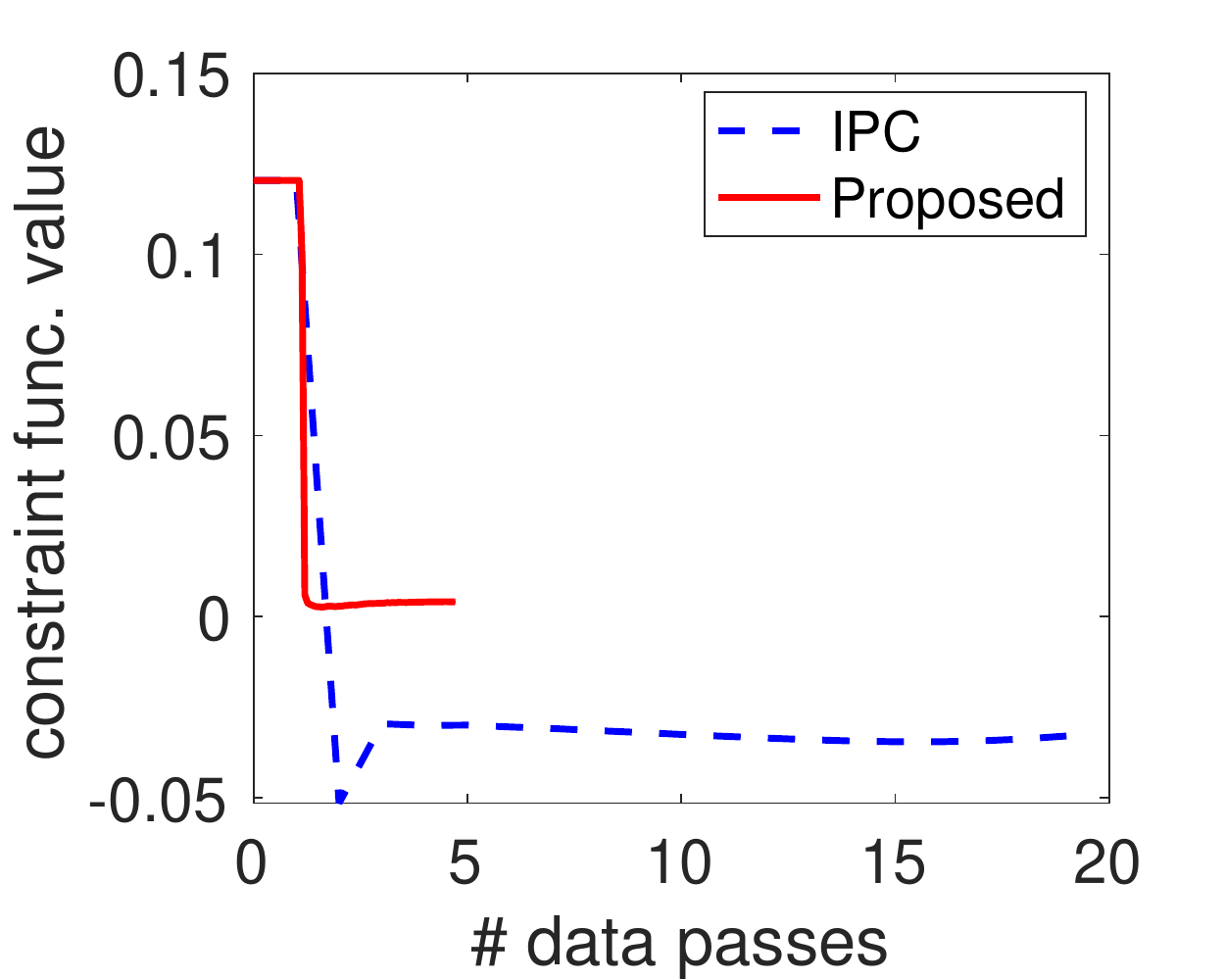} & 
			\includegraphics[width=0.35\textwidth]{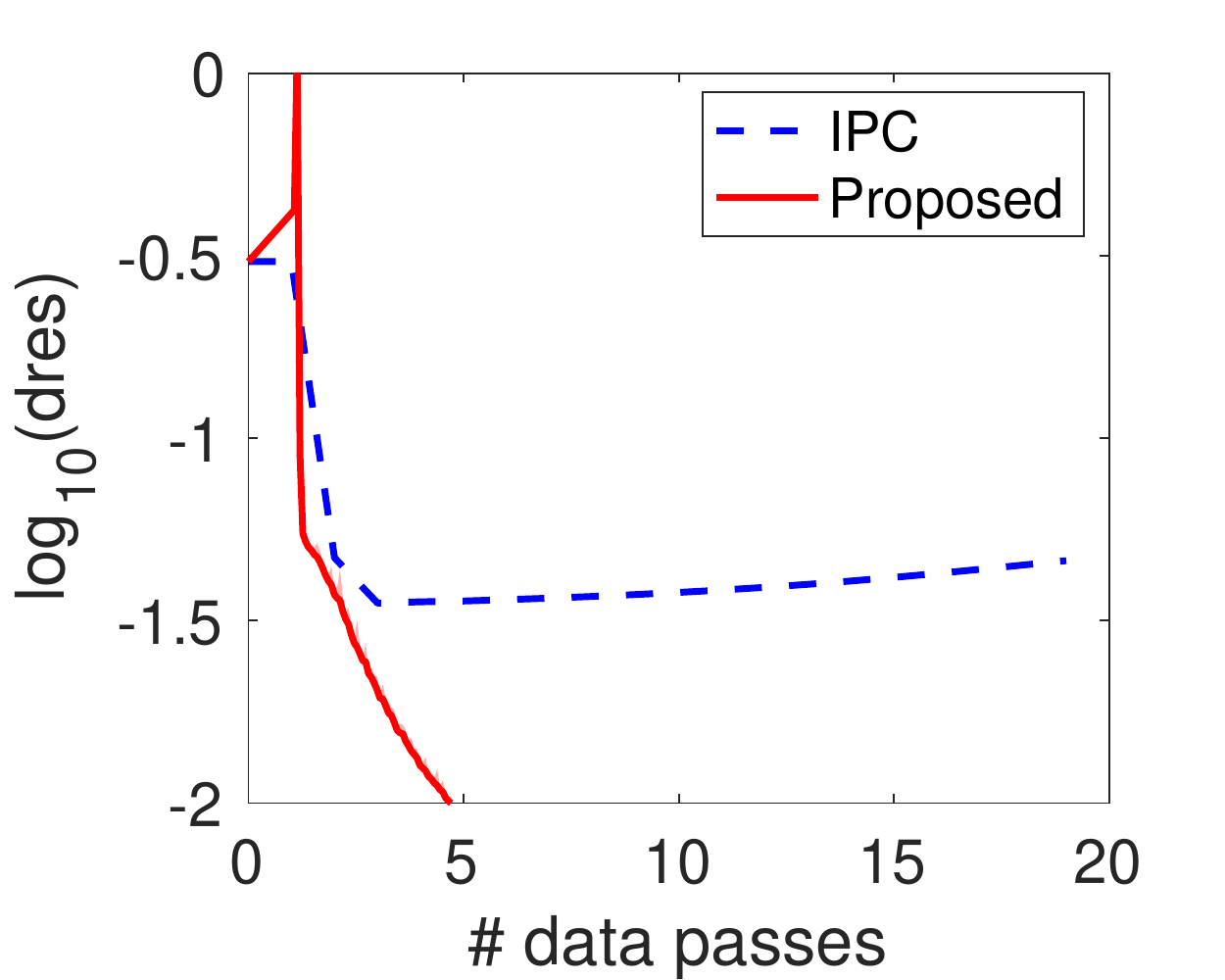} \\[0.4cm]
			\multicolumn{2}{c}{{\Large\emph{a9a} dataset}}\\
	\includegraphics[width=0.35\textwidth]{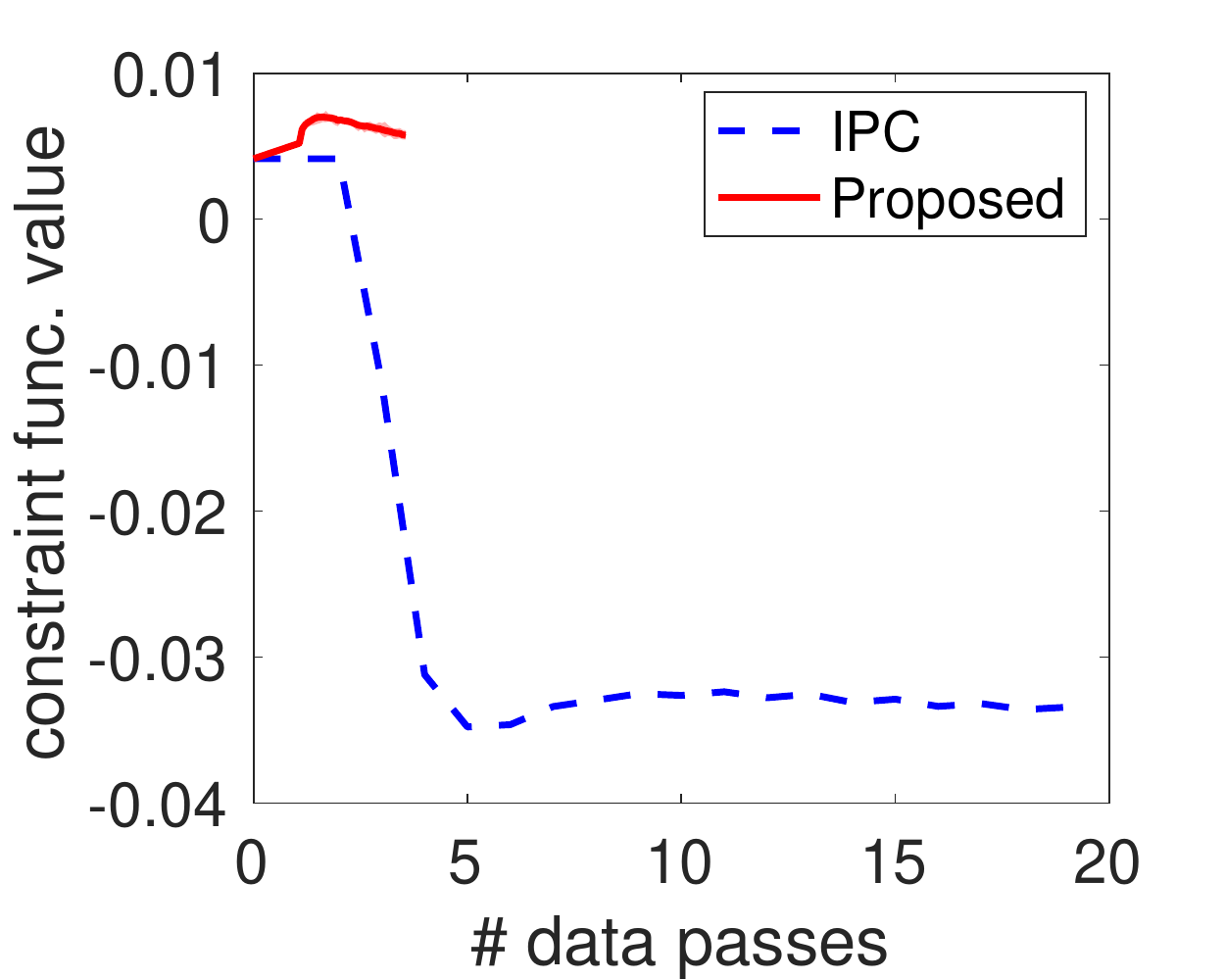} & 
			\includegraphics[width=0.35\textwidth]{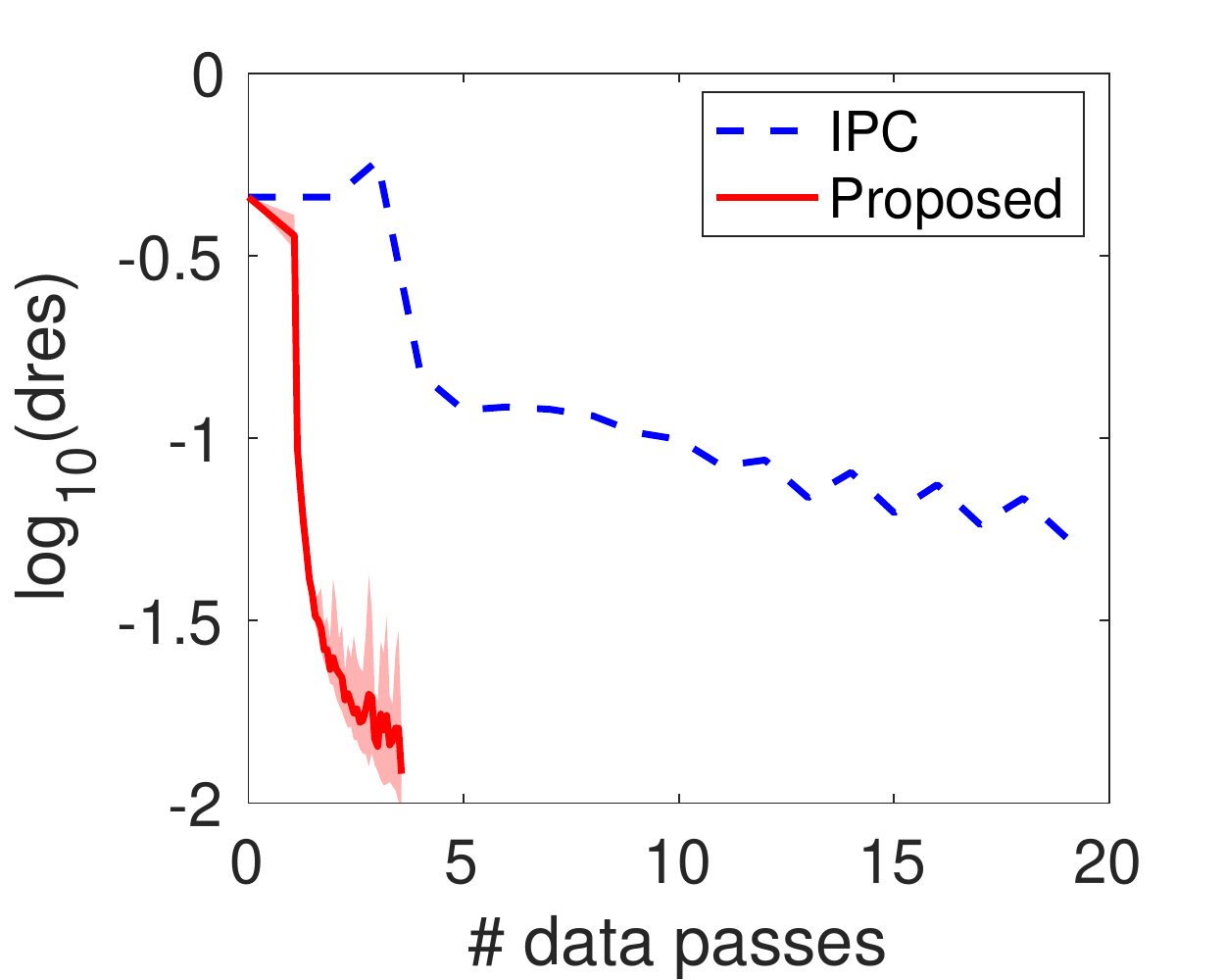} \\[0.4cm]
			\multicolumn{2}{c}{{\Large\emph{loan} dataset}}\\
	\includegraphics[width=0.35\textwidth]{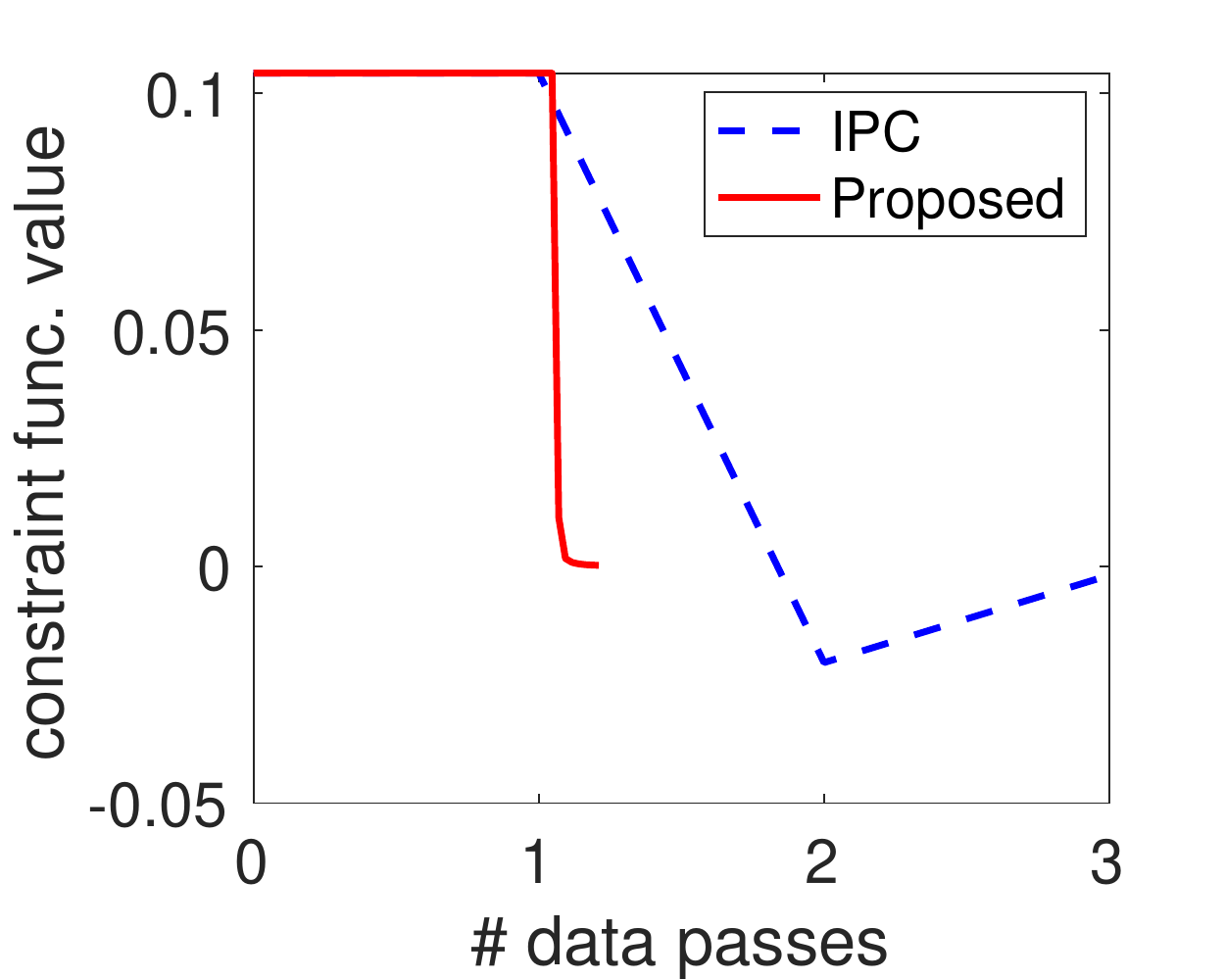} & 
			\includegraphics[width=0.35\textwidth]{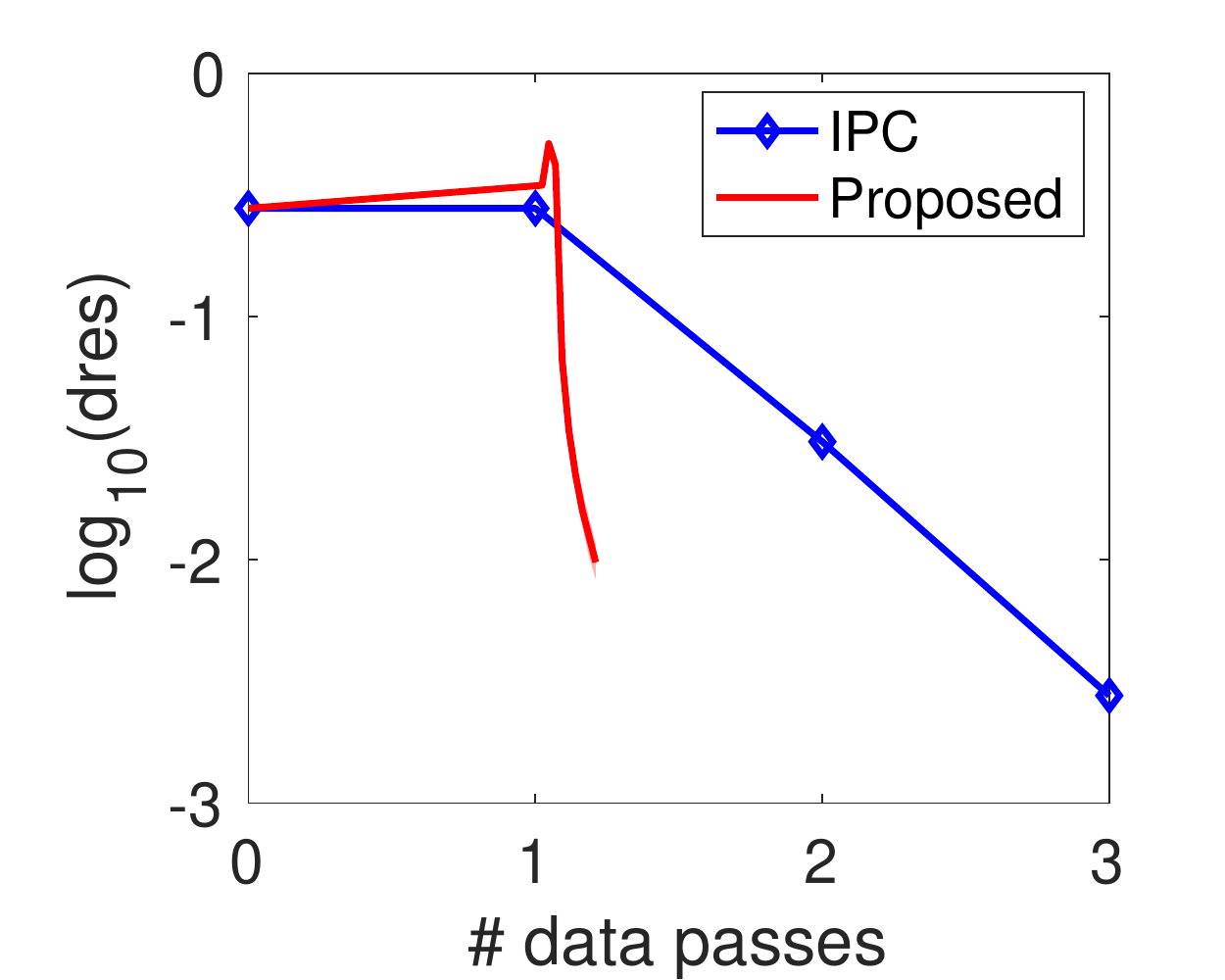}
		\end{tabular}
	\end{center}
	\vspace{-0.2cm}
		\caption{Mean curves and standard deviation (by shadow area$^\dagger$) of 10 different trials by the proposed Stoc-iALM and the curves by the IPC in \cite{ma2019proximally} on solving the fairness constrained problem \eqref{eq:fair} with \emph{bank}, \emph{a9a}, and \emph{loan} data sets (from top to bottom).\\[0.2cm]
$^\dagger$ Some shadow areas are invisible because the deviations are too small.}\label{fig:fair_10seeds}
\end{figure}

\subsection{Nonconvex Neyman-Pearson Classification}
In this subsection, we test our proposed Stoc-iALM (Algorithm \ref{alg:ialm-stoc}) with the PStorm (Algorithm \ref{alg:pstorm}) subroutine on solving the nonconvex Neyman-Pearson classification problem \cite{yan2022adaptive, rigollet2011neyman}. The problem aims at minimizing the false-negative error subject to a constraint on the level of false-positive error. It can be formulated as
\begin{equation}\label{eq:np-clas}
	\begin{aligned}
	\min_{\vx \in \RR^d} & \ f_0(\vx) := \frac{1}{n^+} \sum_{i=1}^{n^+} \phi(\vx^{\top}\va_i^+) , \\ 
	\st & \ f_1(\vx) := \frac{1}{n^-} \sum_{i=1}^{n^-} \phi(-\vx^{\top}\va_i^-) - \hat{c} \le 0,
	\end{aligned}
\end{equation}
	where $\{\va_i^+\}_{i=1}^{n^+}$ and $\{\va_i^-\}_{i=1}^{n^-}$ denotes the positive-class samples and negative-class samples of the training data set. The parameter $\hat{c}$ controls the level of the false-positive error. 
	In \eqref{eq:np-clas}, we set $\phi(\cdot)$ to the sigmoid function: $\phi(u) = 1/(1+\exp(u))$. 
	We use three data sets: \emph{spambase} \cite{Dua:2019} with $d = 57$ and $(n^+,n^-) = (1813, 2788)$, \emph{madelon} \cite{guyon2004result} with $d = 500$ and $(n^+,n^-) = (1300, 1300)$, and \emph{gisette} \cite{guyon2004result} with $d = 2000$ and $(n^+,n^-) = (3500, 3500)$. 
To make sure the feasibility of the problem, we set the false-positive error parameter to $\hat{c} = 0.2$ for \emph{spambase} and \emph{gisette}, and $\hat{c} = 0.4$ for \emph{madelon}. Following \cite{yan2022adaptive}, before feeding each data set into the solvers, we preprocess them by first normalizing it feature-wisely to have mean $0$ and variance $1$, and then scaling each sample to have unit $2$-norm.

Similar to Section \ref{subsec:num_fair}, we reformulate the inequality constraint in \eqref{eq:np-clas} to an equality constraint $f_1(\vx)+\vv = 0$ for our method Stoc-iALM, where $\vv \ge 0$ is enforced. The compared IPC method in \cite{ma2019proximally} is applied directly to \eqref{eq:np-clas}. 
The tolerance is again set to $\vareps = 10^{-2}$ in all tests. Both methods are terminated if the violation of primal and dual feasibility 
is below $\vareps$, where the dual variable of the IPC is computed by \eqref{eq:opt-dual} as in Section \ref{subsec:num_fair}. For Stoc-iALM, at the $k$-th outer iteration, we set $\beta_k = 2^k$ and the smoothness constant to $\frac{\beta_k + 1}{2}$, and again we set $\tilde\vc(\vx^{k+1}) = \vc(\vx^{k+1})$ in the $\vy$-update \eqref{eq:alm-y-stoc}. In the PStorm subroutine, we set the mini-batch size to $10$ for \emph{spambase}, and $30$ for \emph{madelon} and \emph{gisette}. 
The parameter settings of the IPC exactly follow from the code of \cite{ma2019proximally} provided by its authors. Again, we record the primal and dual residuals after every 50 inner iterations in Stoc-iALM, and after every data pass in IPC. Both methods start from a zero vector for each data set, and we perform 10 independent trials by using different random seeds for the proposed method.

Table~\ref{table:np_10seeds} gives \verb|pres| and \verb|dres| at the produced $\vareps$-KKT point and \verb|#data| by each method to produce such a point. Figure~\ref{fig:np_10seeds} plots the curves of the constraint function value and \verb|dres|. Again, we see that our proposed method Stoc-iALM needs significantly fewer data passes to produce a KKT point with the same-level error tolerance.

\begin{table}\caption{The violation of primal feasibility and the violation to the dual feasibility at the produced $\vareps$-KKT point with $\vareps=10^{-2}$, and the number of data passes (shortened by pres, dres and \#data$^\dagger$ respectively) of 10 trials with different random seeds by the proposed Stoc-iALM and the IPC in \cite{ma2019proximally} on solving the Neyman-Pearson classification problem \eqref{eq:np-clas} with \emph{spambase}, \emph{madelon}, and \emph{gisette} data sets (from left to right).\\[0.2cm]
$^\dagger$ \#data by Stoc-iALM are fractional because the subroutine PStorm uses minibatch of data points to compute sample gradients and we record pres and dres after every 50 inner iterations.}\label{table:np_10seeds} 
\begin{center}
\resizebox{1 \textwidth}{!}{ 
\begin{tabular}{|c||c|c|c|} 
\hline
method [\emph{spambase}] & pres & dres & \#data \\\hline\hline 
Stoc-iALM (1) & 0 & 7.8e-3 & 18.75  \\ 
Stoc-iALM (2) & 0 & 6.1e-3 & 39.23  \\ 
Stoc-iALM (3) & 0 & 8.5e-3 & 13.74  \\ 
Stoc-iALM (4) & 0 & 7.6e-3 & 16.93  \\ 
Stoc-iALM (5) & 0 & 9.0e-3 & 11.01  \\ 
Stoc-iALM (6) & 0 & 6.5e-3 & 21.48  \\ 
Stoc-iALM (7) & 0 & 6.8e-3 & 28.76  \\ 
Stoc-iALM (8) & 0 & 9.9e-3 & 22.85  \\ 
Stoc-iALM (9) & 0 & 9.3e-3 & 11.47  \\ 
Stoc-iALM (10) & 0 & 7.1e-3 & 16.47  \\ 
\hline
IPC & 0 & 9.7e-3 & 37  \\\hline 
\end{tabular}

\begin{tabular}{|c||c|c|c|} 
\hline
method [\emph{madelon}] & pres & dres & \#data \\\hline\hline 
Stoc-iALM (1) & 0 & 1.00e-2 & 336.08  \\ 
Stoc-iALM (2) & 0 & 1.00e-2  & 324.77  \\ 
Stoc-iALM (3) & 0 & 1.00e-2 & 330.08  \\ 
Stoc-iALM (4) & 0 & 1.00e-2  & 314.15  \\ 
Stoc-iALM (5) & 0 & 1.00e-2  & 323.15  \\ 
Stoc-iALM (6) & 0 & 1.00e-2 & 322.92  \\ 
Stoc-iALM (7) & 0 & 1.00e-2  & 317.85  \\ 
Stoc-iALM (8) & 0 & 1.00e-2  & 332.38  \\ 
Stoc-iALM (9) & 0 & 1.00e-2  & 326.38  \\ 
Stoc-iALM (10) & 0 & 1.00e-2  & 341.15  \\ 
\hline
IPC & 0 & 1.00e-2  & 1804  \\\hline 
\end{tabular}

\begin{tabular}{|c||c|c|c|} 
\hline
method [\emph{gisette}] & pres & dres & \#data \\\hline\hline 
Stoc-iALM (1) & 0 & 1.00e-2 & 234.04  \\ 
Stoc-iALM (2) & 0 & 1.00e-2 & 235.24  \\ 
Stoc-iALM (3) & 0 & 1.00e-2 & 234.04  \\ 
Stoc-iALM (4) & 0 & 1.00e-2 & 235.24  \\ 
Stoc-iALM (5) & 0 & 1.00e-2 & 234.04  \\ 
Stoc-iALM (6) & 0 & 1.00e-2 & 235.24  \\ 
Stoc-iALM (7) & 0 & 1.00e-2 & 235.24  \\ 
Stoc-iALM (8) & 0 & 1.00e-2 & 235.24  \\ 
Stoc-iALM (9) & 0 & 1.00e-2 & 234.04  \\ 
Stoc-iALM (10) & 0 & 1.00e-2 & 234.04  \\ 
\hline
IPC & 0 & 1.00e-2 & 650  \\\hline 
\end{tabular}
}
\end{center}
\end{table}

\begin{figure}[t] 
	\begin{center}
		\begin{tabular}{cc}
\multicolumn{2}{c}{{\Large\emph{spambase} dataset}}  \\
    \includegraphics[width=0.35\textwidth]{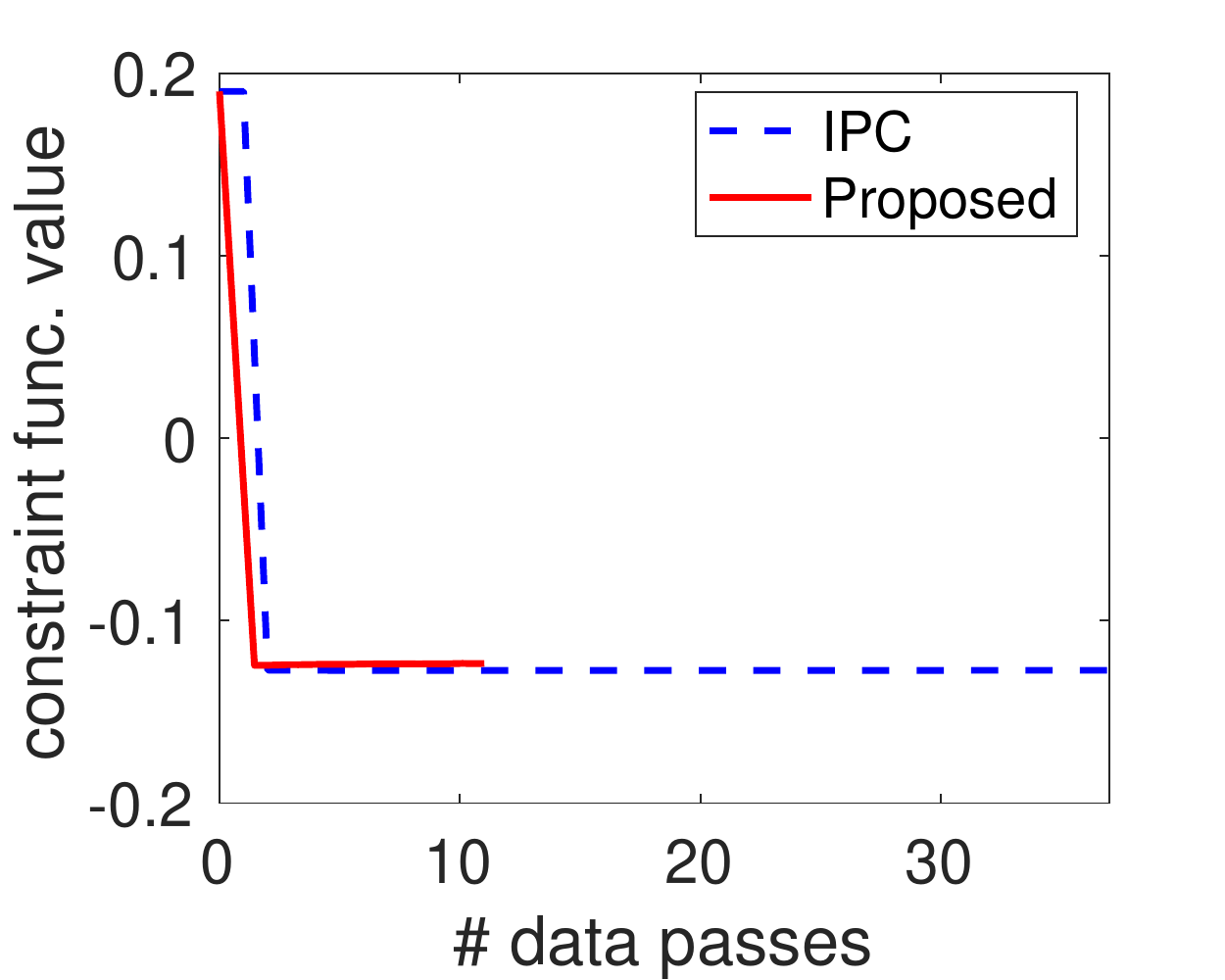} &
			\includegraphics[width=0.35\textwidth]{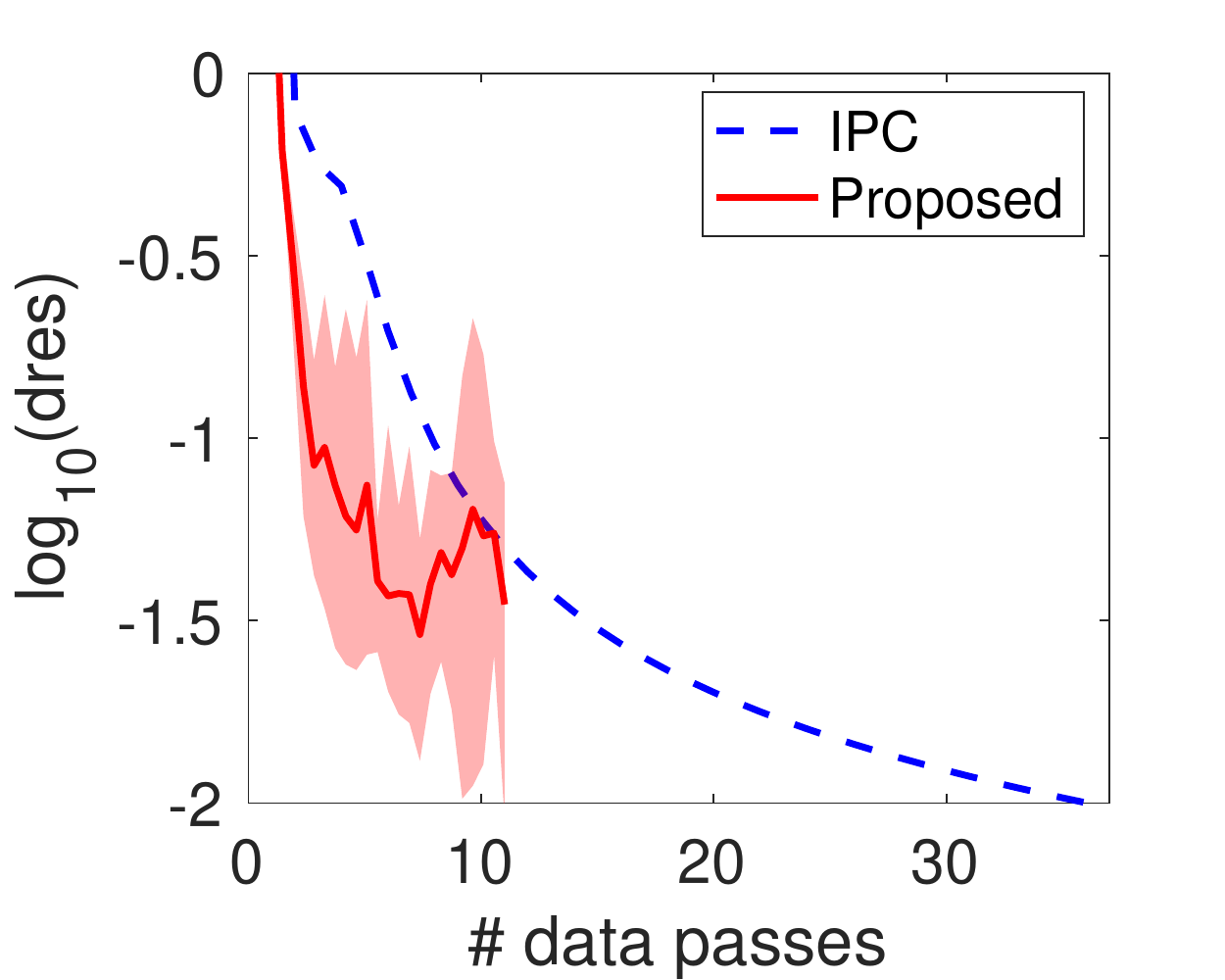} \\[0.4cm]
			\multicolumn{2}{c}{{\Large\emph{madelon} dataset}}\\
	\includegraphics[width=0.35\textwidth]{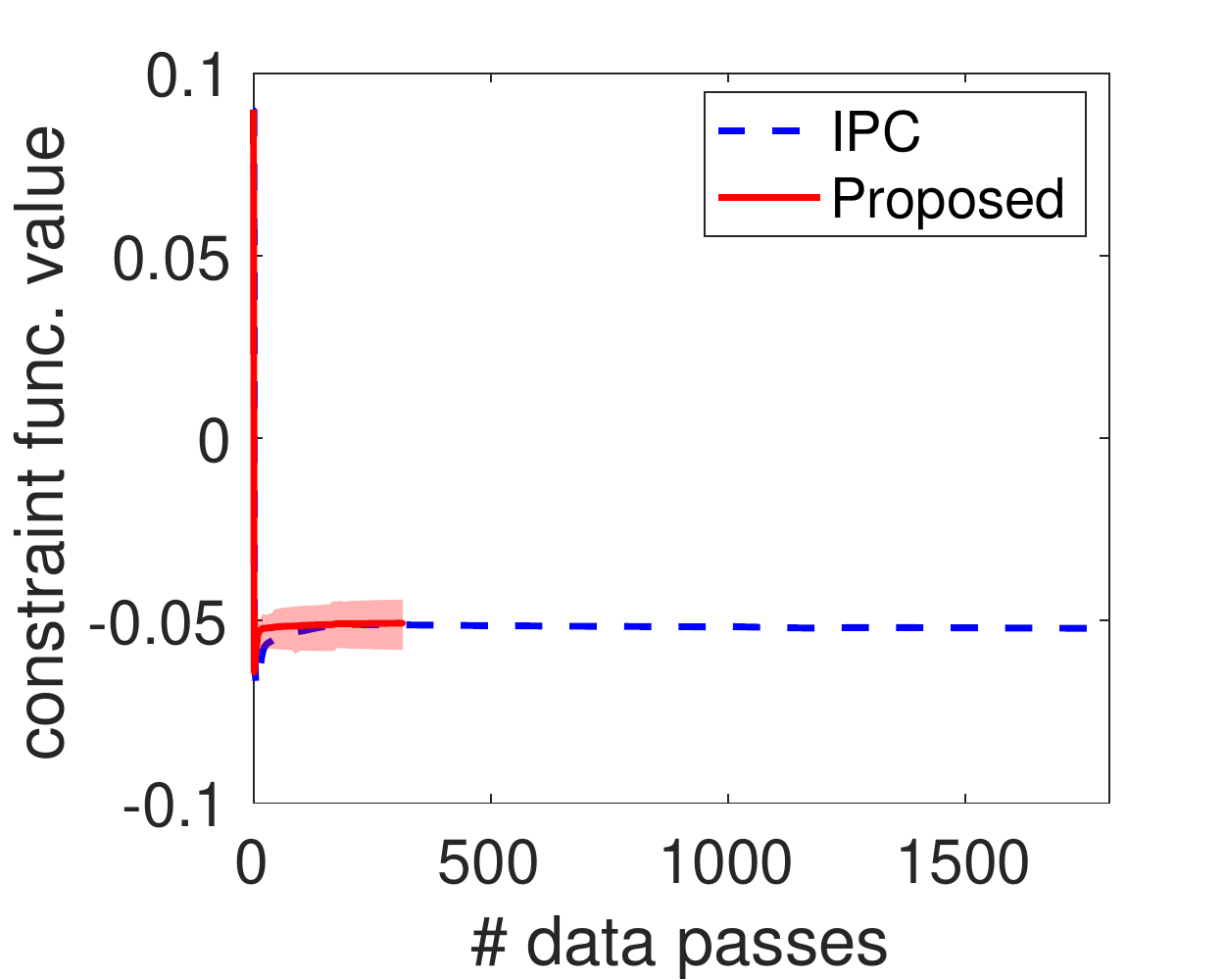} & 
			\includegraphics[width=0.35\textwidth]{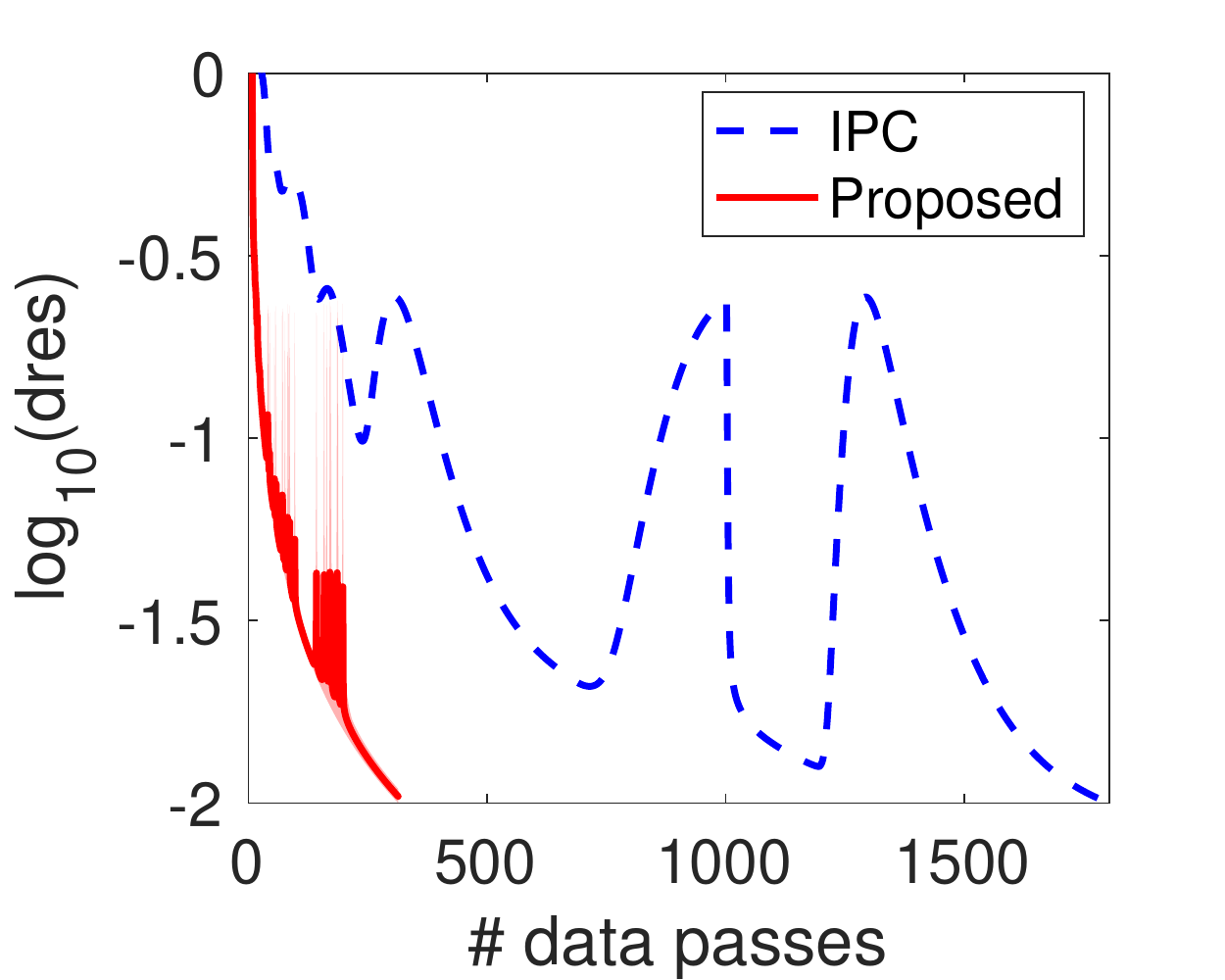} \\[0.4cm]
			\multicolumn{2}{c}{{\Large\emph{gisette} dataset}}\\
	\includegraphics[width=0.35\textwidth]{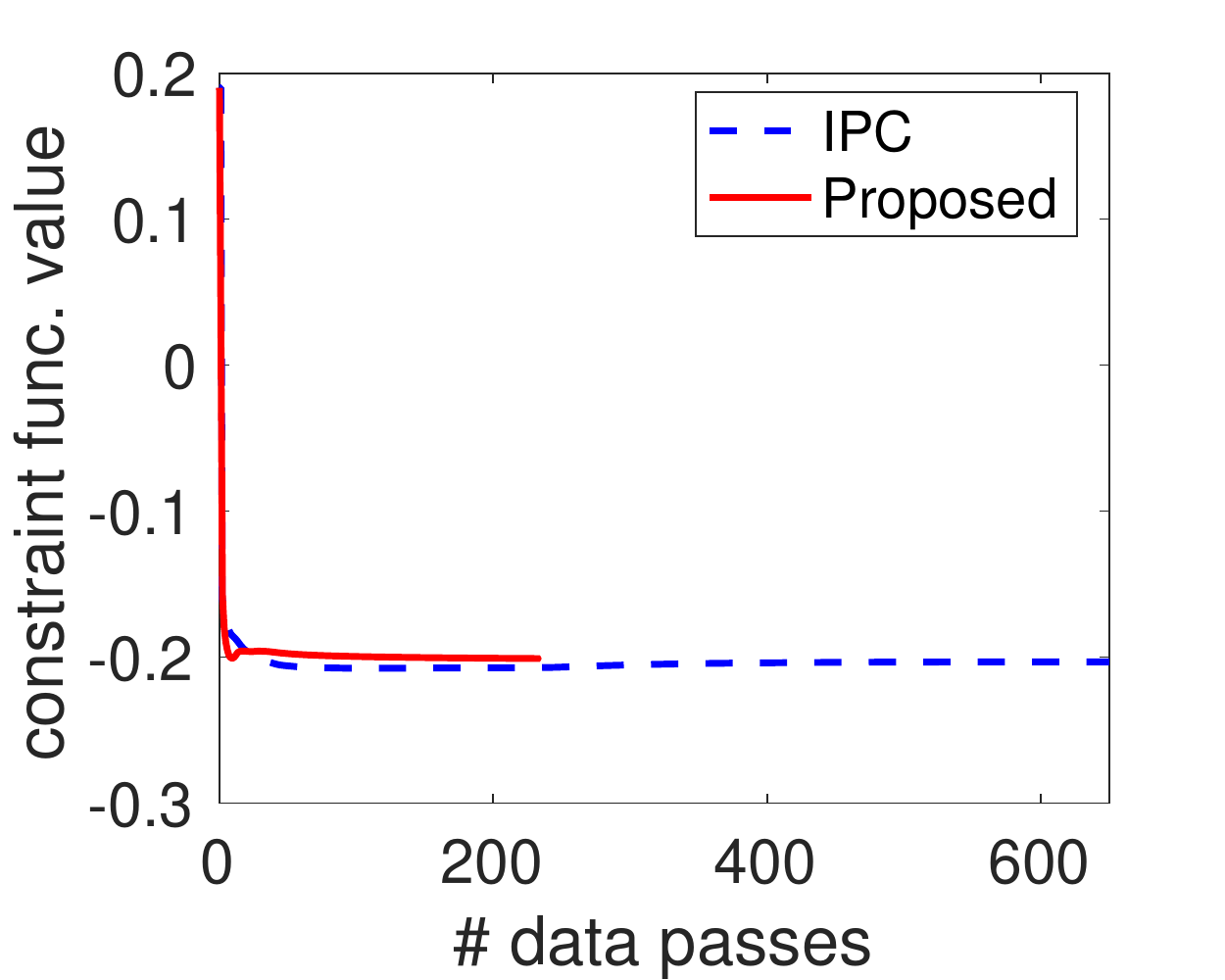} & 
			\includegraphics[width=0.35\textwidth]{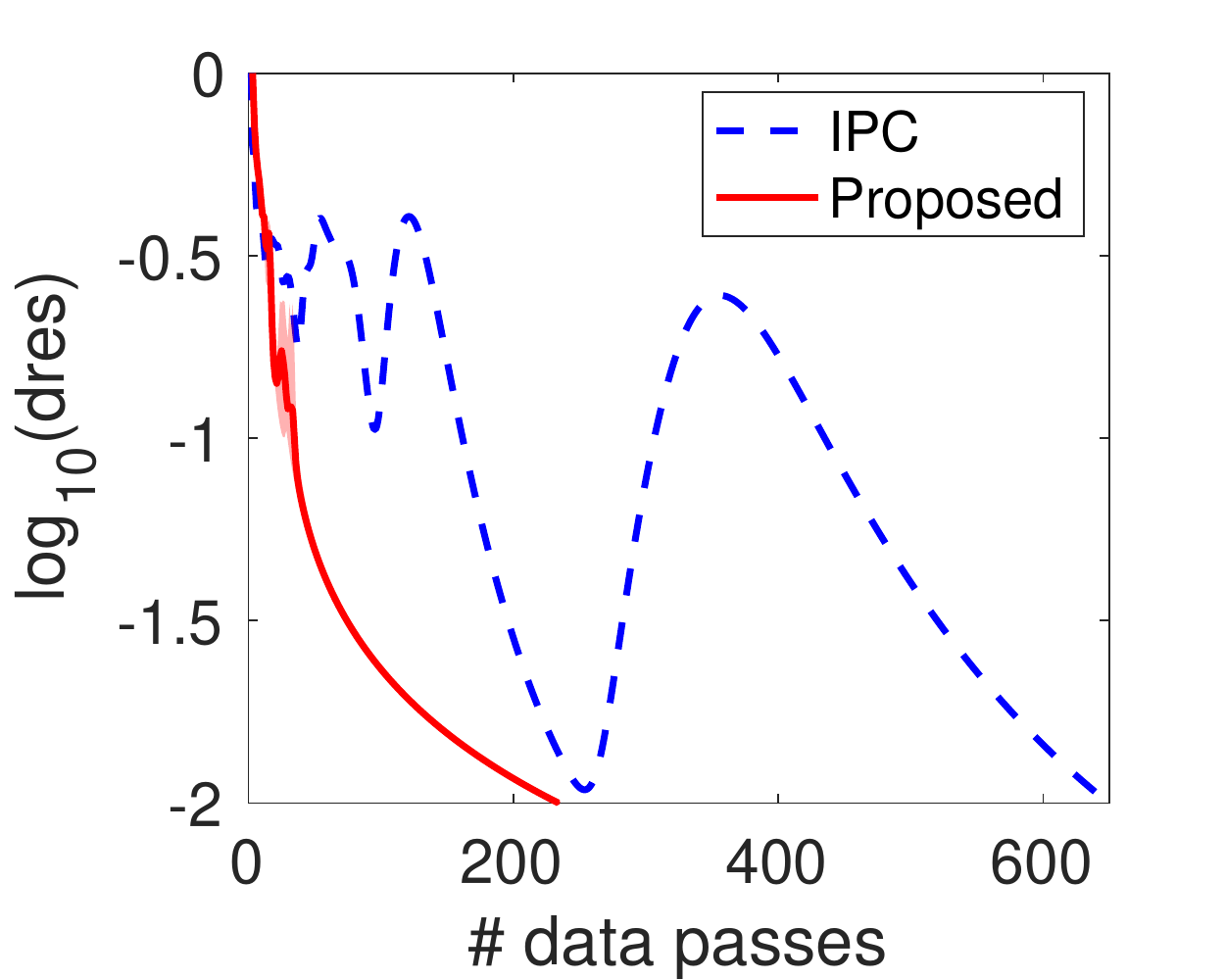}
		\end{tabular}
	\end{center}
	\vspace{-0.2cm}
		\caption{Mean curves and standard deviation (by shadow area$^\dagger$) of 10 different trials by the proposed Stoc-iALM and the curves by the IPC in \cite{ma2019proximally} on solving the Neyman-Pearson classification problem \eqref{eq:np-clas} with \emph{spambase}, \emph{madelon}, and \emph{gisette} data sets (from top to bottom).\\[0.2cm]
		$^\dagger$ Some shadow areas are invisible because the deviations are too small.}\label{fig:np_10seeds}
\end{figure}

\section{Conclusion}
We have presented a stochastic inexact augmented Lagrangian method (Stoc-iALM) for solving nonconvex expectation constrained optimization.  
To handle nonconvex stochastic iALM subproblems, we apply a momentum-based variance-reduced proximal stochastic gradient method (PStorm) subroutine with a proposed post-processing step. 
To reach an $\vareps$-KKT solution in expectation, we establish an oracle complexity of $O(\vareps^{-5})$, which improves over the state-of-the-art complexity of $O(\vareps^{-6})$. 
Numerically, we have demonstrated that the proposed Stoc-iALM can significantly outperform one state-of-the-art method. 

\section*{Acknowledgements} This work is partly supported by NSF grants DMS-2053493 and DMS-2208394 and the ONR award N00014-22-1-2573, and also by the Rensselaer-IBM AI Research Collaboration, part of the IBM AI Horizons Network.	

	
\bibliographystyle{abbrv} 

\end{document}